\documentclass{amsart}

\usepackage{amsmath,amssymb,amsbsy,amsfonts,amsthm,latexsym,
                       amsopn,amstext,amsxtra,euscript,amscd}
\usepackage{enumerate}
\usepackage{array}
\usepackage{ifthen}
\usepackage{url}
\usepackage{color}

\newtheorem{theorem}{Theorem}
\newtheorem{lemma}{Lemma}
\newtheorem{proposition}{Proposition}

\newtheorem{conjecture}{Conjecture}
\newtheorem*{definition}{Definition}

\theoremstyle{definition}
\newtheorem{remark}{Remark}[section]

\newcommand{\klammern}[4][]%
{\ifthenelse{\equal{#1}{}}{\left#2}{\csname#1\endcsname#2}%
#4\ifthenelse{\equal{#1}{}}{\right#3}{\csname#1\endcsname#3}}
\newcommand{\betrag}[2][]{\klammern[#1]{\lvert}{\rvert}{#2}}

\newcommand{\conj}[1]{^{(#1)}}

\def\QQ{\mathbb Q}
\def\ZZ{\mathbb Z}
\def\RR{\mathbb R}

\begin{document}

\title{There is no Diophantine quintuple}

\author{Bo He}
\address{Department of Mathematics, Hubei University for Nationalities, Enshi, Hubei, 445000 P.R. China and Institute of Applied Mathematics,  Aba Teachers University, Wenchuan, Sichuan, 623000 P. R. China}
\email{bhe@live.cn}

\author{Alain Togb\'e}
\address{Department of Mathematics, Statistics, and Computer Science, 
Purdue University Northwest, 1401 S, U.S. 421, Westville IN 46391 USA}
\email{atogbe@pnw.edu}

\author{Volker Ziegler}
\address{Institute of mathematics, University of Salzburg, Hellbrunner Strasse 34/I, A-5020 Salzburg, Austria}
\email{volker.ziegler@sbg.ac.at}




\begin{abstract}
A set of $m$ positive integers $\{a_1, a_2, \dots , a_m\}$ is called a Diophantine $m$-tuple if $a_i  a_j + 1$ is a perfect square for all $1 \le i < j \le m$. 
In~\cite{Dujella:2004} Dujella proved that there is no Diophantine sextuple and that there are at most finitely many Diophantine quintuples.
In particular, a folklore conjecture concerning Diophantine $m$-tuples states that no Diophantine quintuple exists at all. In this paper we prove this conjecture. 
\end{abstract}

\date{\today}

\maketitle

\tableofcontents

\section{Introduction}\label{sec:1}

A set of $m$ distinct, positive integers $\{a_1, \dots,  a_m\}$ is called a Diophantine $m$-tuple if $a_i a_j +1$ is a perfect square for all $1 \le i < j \le m$.
Diophantus~\cite{Diophantus} (see also~\cite{Dickson}) studied sets of positive rational numbers with the same property, in particular he found the set of four
positive rational numbers $\left\{\frac{1}{16}, \frac{33}{16}, \frac{17}{4}, \frac{105}{16}\right\}$. But, the first example for an integral Diophantine quadruple $\{1,3,8,120\}$
was found by Fermat. Later Euler showed that there exist infinitely many Diophantine quadruples.

One of the first key results concerning the possible existence of Diophantine
quintuples was achieved in 1969, by Baker and Davenport~\cite{Baker-Davenport:1969} who proved that the fourth element $120$ in Fermat's
quadruple uniquely extends the Diophantine triple  $\{1, 3, 8\}$.
Thus, they showed that the Diophantine quadruple $\{1, 3, 8, 120\}$ cannot be extended by a fifth positive integer to a Diophantine quintuple.
However, no Diophantine quintuple was yet found and it is a folklore conjecture, the so-called {\it Diophantine quintuple conjecture}, that no Diophantine quintuple exists.
In this paper, we give a proof of this conjecture. 

\begin{theorem}\label{thm:main}
There does not exist a Diophantine quintuple.
\end{theorem} 
 
In view of finding Diophantine quintuples or showing that none exists, one of the most important topics is the extensibility and existence of Diophantine $m$-tuples.
For any fixed pair of positive integers $a$ and $b$ such that $ab+1=r^2$ is a perfect square, i.e. $\{a, b\}$ is a Diophantine pair,
Euler proved that one can always add to $\{a, b\}$ a third element, namely $a+b+2\sqrt{ab+1}$, in order to obtain a Diophantine triple of the form
\begin{equation}\label{eq:euler-tip}
\{a,b,a+b+2r\}.
\end{equation}
In fact, for every Diophantine pair $\{a, b\}$ there exist infinitely many positive integers $c$ such that $\{a, b, c\}$ is a Diophantine triple.
Moreover, Euler observed that adding $4r(a+r)(b+r)$ to the Diophantine triple \eqref{eq:euler-tip} one obtains a Diophantine quadruple 
\begin{equation}\label{eq:euler-qud}
\{a,b,a+b+2r, 4r(a+r)(b+r)\}
\end{equation}
therefore proving the existence of infinitely many Diophantine quadruples. We call a Diophantine triple of form \eqref{eq:euler-tip}
an \emph{Euler triple}, or a \emph{regular triple} (see \cite{Fujita-Miyazaki}) and a Diophantine quadruple of form
\eqref{eq:euler-qud} an \emph{Euler quadruple}, or \emph{doubly regular quadruple} (see \cite{Martin-Sitar}).

Let $\{a,b,c\}$ be a Diophantine triple (not necessarily of form \eqref{eq:euler-tip}), i.e. $ab+1=r^2$, $ac+1=s^2$ and $bc+1=t^2$ are all perfect squares.
In 1979, Arkin, Hoggatt and Strauss~\cite{Arkin-Hoggatt-Strauss:1979} noticed that adding  
\begin{equation*}
\begin{split}
d_{+}&=a+b+c+2abc+2\sqrt{(ab+1)(ac+1)(bc+1)}\\
& = a+b+c+2abc+2rst
\end{split}
\end{equation*}
 to the Diophantine triple $\{a,b,c\}$ results in a Diophantine quadruple of the form
\begin{equation*}
\{a,b,c, a+b+c+2abc+2rst\}.
\end{equation*} 
Such a quadruple is called a \emph{regular Diophantine quadruple}. Consequently, Euler quadruples are a special case of regular quadruples.
Since all known Diophantine quadruples are regular, several authors (see e.g. \cite{Arkin-Hoggatt-Strauss:1979}, \cite{Dujella-Pethoe:1998})
were led to an even stronger version of Theorem~\ref{thm:main} which is still open.

\begin{conjecture}\label{conj:2}
 If $\{a, b, c, d\}$ is a Diophantine quadruple such that $d > \max\{a, b, c\}$, then $d = d_+$.
\end{conjecture}

 After the result of Baker and Davenport on the extensibility of the triple $\{1,3,8\}$, the extensibility of further, generalized pairs and triples was studied.
 All known results positively support Conjecture~\ref{conj:2} (see Table~\ref{tbl:1} below). In addition, in a series of papers \cite{He-Togbe:1,He-Togbe:2,He-Togbe:3}, He and Togb\'e verified Conjecture~\ref{conj:2} for triples of the form $\{k,A^2k+2A, (A+1)^2k+2(A+1)\}$ with two parameters $k$ and $A$, where $2\le A\le 10$ or $A\ge 52330$. In particular, they showed that such a triple can be extended only to a regular quadruple. For more results about the regularity of Diophantine pairs and triples, we refer to \cite{Dujella:1999} and \cite{FFT1}. Recently, Fujita and Miyazaki~\cite{Fujita-Miyazaki} proved that any fixed Diophantine triple can be extended only to a Diophantine quadruple in at most $11$ ways by joining a fourth element exceeding the maximal element in the triple.
 
\begin{table}[!h]\label{tbl:1}
\begin{center}
\def\temptablewidth{0.6\textwidth}
{\rule{\temptablewidth}{1 pt}}
\begin{tabular*}{\temptablewidth}{@{\extracolsep{\fill}}cccc}
 {\bf Pairs / Triples } & {\bf References}   \\   \hline
 $\{k-1, k+1, 4k\}$  & \cite{Dujella:1997-1}   \\   \hline
  \{1, 3\}        & \cite{Dujella-Pethoe:1998}     \\ \hline
 $\{k-1,k+1\}$        &     \cite{Fujita,BDM}   \\ \hline
 $\{a,b\}$ with $b<a+4\sqrt{a}$   &  \cite{FFT2}  \\ \hline
  $\{k, 4k\pm 4\}$       &   \cite{FFT2,HPST}  \\ \hline
  $\{1, b\}$ with $b-1$ be prime       &   \cite{HPTY}  
              \end{tabular*}
       {\rule{\temptablewidth}{1pt}}
       \end{center}
    \tabcolsep 0pt \caption{Results on the extensibility of Diophantine pairs and triples} \vspace*{-12pt}
       \end{table}

Obviously Conjecture~\ref{conj:2} implies Theorem~\ref{thm:main}, but a proof of Conjecture~\ref{conj:2} seems to be still out of reach. We do not
even know whether there are infinitely many irregular Diophantine quadruples or not. In the case of the Diophantine quintuple conjecture, Theorem \ref{thm:main}, several researchers obtained several important results. Particularly, the first absolute bound for the size of Diophantine $m$-tuples was given by Dujella~\cite{Dujella:2001}, when he showed that $m\le 8$. In 2004, Dujella~\cite{Dujella:2004} proved that there does not exist a Diophantine sextuple. Moreover, Dujella obtained the following result:

\begin{theorem}\label{thm:dujella}
 There are only finitely many Diophantine quintuples.
 \end{theorem}
 
 A further important step towards a proof of Theorem \ref{thm:main} was made by Fujita~\cite{Fujita-1}, who proved the following:
 
 \begin{theorem}\label{thm:fujita}
 If $\{a,b,c,d,e\}$ is a Diophantine quintuple with $a<b<c<d<e$, then $d=d_{+}$.
 \end{theorem}
 
 A good estimate for upper bounds of $d$ is a key step of the settlement of Theorem~\ref{thm:main}. The following table contains a summary of the progress made towards Theorem~\ref{thm:main}. 
 
 \begin{table}[!h] \label{tbl:2}
\begin{center}
\def\temptablewidth{0.72\textwidth}
{\rule{\temptablewidth}{1.1 pt}}
\begin{tabular*}{\temptablewidth}{@{\extracolsep{\fill}}ccccc}
 {\bf Authors} & {\bf Year} & {\bf $d\le$}   \\   \hline
   Dujella   \cite{Dujella:2004} & 2004  &$10^{2171} $    \\ \hline
  Fujita   \cite{Fujita-number} &  2010  &$10^{830}$   \\ \hline
  Filipin and Fujita  \cite{FF} & 2013  &$10^{100}$  \\ \hline
  Elsholtz, Filipin and Fujita \cite{EFF} & 2013 &$3.5\cdot 10^{94}$   \\ \hline
    Wu and He \cite{Wu-He:2014}  & 2014 & $10^{74}$ \\  \hline
  Cipu \cite{Cipu}  & 2015   &$10^{72.188}$  \\ \hline
  Trudgian \cite{Trudgian} & 2015 & $4.02\cdot 10^{70}$  \\ \hline
  Cipu, Trudgian \cite{Cipu-Trudgian} & 2016 & $7.228 \cdot 10^{67}$   
              \end{tabular*}
       {\rule{\temptablewidth}{1.1pt}}
       \end{center}
    \tabcolsep 0pt \caption{Upper bounds for $d$} \vspace*{-15pt}
       \end{table}  

For a more complete account of Diophantine $m$-tuples and related problems we refer to \cite{Dujella:2016} or Dujella's web page \cite{Dujella:HP}.

\section{Outline of the proof} \label{sec:2}

The three new key arguments that lead to the proof of our main result are:

\begin{enumerate}
\item The definition of an operator on Diophantine triples and their classification. 

\item The use of sharp lower bounds for linear forms in three logarithms obtained by applying
a result due to Mignotte \cite{Mignotte:kit} (see also~\cite{BMS1,BMS2}).

\item The use of new congruences in the case of Euler quadruples.    
\end{enumerate}

The purpose of this section is to give more details to these three key arguments and an outline of the proof of Theorem \ref{thm:main}. After stating some auxiliary results in Section~\ref{sec:aux}, we will define in Section~\ref{sec:3} the $\partial$- and $\partial_{-D}$-operators on Diophantine triples,
where $D$ is a nonnegative integer. The $\partial$-operator is defined for non-Euler triples and yields a new triple which is ``closer'' to the property of being an Euler triple. One of the key results in Section~\ref{sec:3} is that if we apply the $\partial$ operator repeatedly we always arrive at an Euler triple in finite time. This allows us to introduce the \emph{degree} of a Diophantine triple. Roughly speaking an Euler triple has degree $0$ and a triple to which an $D$-fold application of $\partial$ yields an Euler triple has degree $D$. This leads us to a new classification of Diophantine triples. 

In Section \ref{sec:4}, we set up a system of Pell equations associated with a Diophantine quintuple $\{a, b, c, d, e\}$:
\begin{align*}
aY^2 - bX^2 &= a - b,\\
aZ^2 - cX^2 &= a - c,\\
bZ^2 - cY^2 &= b - c,\\
aW^2 - dX^2 &= a - d,\\
bW^2 - dY^2 &= b - d,\\
cW^2 - dZ^2 &= c - d. 
\end{align*}
Using ideas due to Fujita \cite{Fujita-number},  we can show that the solutions to this system of Pell equations satisfy
\begin{align*}
Y \sqrt{a} + X \sqrt{b} &= (\sqrt{a} + \sqrt{b})(r+\sqrt{ab})^{2h},\\
Z \sqrt{a} + X \sqrt{c} &= ( \sqrt{a} + \sqrt{c})(s+\sqrt{ac})^{2j},\\
Z \sqrt{b} + Y \sqrt{c} &= (\sqrt{b} +  \sqrt{c})(t+\sqrt{bc})^{2k},\\
W \sqrt{a} + X \sqrt{d} &= (\varepsilon \sqrt{a} + \sqrt{d})(x+\sqrt{ad})^{2l},\\
W \sqrt{b} + Y \sqrt{d} &= (\varepsilon \sqrt{b} + \sqrt{d})(y+\sqrt{bd})^{2m},\\
W \sqrt{c} + Z \sqrt{d} &= (\varepsilon \sqrt{c} +  \sqrt{d})(z+\sqrt{cd})^{2n},
\end{align*}
for some nonnegative integers $h, j, k, l, m$ and $n$. By extending the classical gap principles first introduced by Dujella and Peth\H{o} \cite{Dujella-Pethoe:1998}, we find several relations and lower bounds for those exponents, most important the inequality $h>6.2 \sqrt{ac}$ (see Lemma \ref{lem:hbd}, in Section \ref{sec:5}).

The next step (Section \ref{sec:6}) is to find an upper bound for $h$ by using Baker's method, i.e. using lower bounds for linear forms
in logarithms. In view of some new results proved in Section \ref{sec:5} we find a slight improvement of the latest result due to Cipu
and Trudgian \cite{Cipu-Trudgian}. Using deep results on lower bounds for linear forms in three logarithms due to Mignotte \cite{Mignotte:kit}
(see also \cite{BMS1,BMS2}) we can further improve the bound and finally arrive at the upper bounds $d<1.83 \cdot 10^{52}$ and $h<5.136 \cdot 10^{13}$. In view of an automatic computer verification these bounds still seem to be too large and some new idea is needed to complete the proof of Theorem~\ref{thm:main}.

This new idea is presented in Section \ref{sec:7}, where we prove new congruences in the case that the Diophantine quintuple $\{a,b,c,d,e\}$ contains an Euler triple. In particular, we show that at least one of the following congruences is satisfied under the hypotheses that $\{a,b,c\}$ is an Euler triple:
\begin{itemize}
 \item $l\equiv n\equiv 0 \mod s$,
 \item $m\equiv n\equiv 0 \mod t$,
 \item $n\equiv \pm r \mod st$.
\end{itemize}
With an application of Laurent's result \cite{Laurent:2008} on lower bounds for linear forms in two logarithms, the first two congruences yield $s,t<22023$ respectively (see Lemmas \ref{lem:euler1} and \ref{lem:euler2}). And the results obtained in Section \ref{sec:6} yield $r<900154$ provided that $n\equiv \pm r \mod st$. These upper bounds are small enough to use a variant of the Baker-Davenport reduction method. Thus, we may conclude that an Euler triple cannot be extended to a Diophantine quintuple.

In the case that the Diophantine triple $\{a,b,c\}$ is of degree $1$, we use again the upper bounds obtained in Section \ref{sec:6} and obtain that
$r<2315167$ and $a<93596$. These bounds are again small enough to check case by case that no Diophantine quintuple exists with the assistance of a computer. 

So we are left to the case that $\{a, b, c\}$ is of degree at least two. Here, we apply the $\partial$ operator to the triple
$\{a, b, c\}$ at least two times which yields a new triple $\{a',b',c'\}$ which is much ``smaller'' than the original one.
By the sharp upper bounds from Section \ref{sec:6} we obtain a feasible number of Diophantine triples that might be extendable to a Diophantine quintuple. Again a computer verification yields that no Diophantine quintuple exists. Thus, also no Diophantine triple of degree $>1$ can be extended to a Diophantine quintuple.

Putting these last three results together, we immediately get our main result, i.e. Theorem \ref{thm:main}.

\section{Auxiliary results}\label{sec:aux}

For a Diophantine triple $\{a, b, c\}$, we define $d_{+}$ and $d_{-}$ by 
$$
d_{+}=d_{+}(a, b, c) = a+b+c + 2abc + 2\sqrt{(ab+1)(ac+1)(bc+1)} 
$$ 
and 
$$
d_{-}=d_{-}(a,b,c) = a+b+c + 2abc -  2\sqrt{(ab+1)(ac+1)(bc+1)}. 
$$
Let $ab+1=r^2$, $ac+1=s^2$ and $bc+1=t^2$, then we have 
$$
{ad_{\pm}+1} =(rs\pm at)^2,\quad {bd_{\pm}+1} = (rt\pm bs)^2,\quad {cd_{\pm}+1} = (cr\pm st)^2. 
$$
Without loss of generality, assume that $a<b<c$. Then we have the following two Lemmas which will be frequently used in this paper
without special reference.

\begin{lemma}[Lemma 4 of  \cite{Jones:1978}]\label{lem:Jones}
 If $\{a,b,c\}$ is a Diophantine triple with $a<b<c$, then $c=a+b+2r$ or $c> 4ab$.
\end{lemma}

\begin{remark}
 The statement of Lemma 4 in \cite{Jones:1978} is slightly different from that given here. In the notation of Jones' paper
 Lemma 4 of \cite{Jones:1978} states that $c> 4c'ab$, where $c'$ is some explicitly given quantity and it is easy to show
 that $c'=0$ if and only if $c=a+b\pm 2r$.
\end{remark}

\begin{lemma}\label{lem:d+ieq}
$4abc+c<d_+(a,b,c)<4abc+4c$.
\end{lemma}

\begin{remark}
 Note that the inequality in Lemma~\ref{lem:d+ieq} was stated by Dujella \cite[page 189]{Dujella:2004}. Since we could not find a proof or reference for this statement.
For the sake of completeness we give here a short proof of this inequality.
\end{remark}

\begin{proof}[Proof of Lemma \ref{lem:d+ieq}]
 The proof of the first inequality is straightforward by noting that
 \begin{align*}
 d_{+}&=a+b+c + 2abc + 2\sqrt{(ab+1)(ac+1)(bc+1)}\\
 &> a+b+c+2abc+2\sqrt{(ab)(ac)(bc)}=a+b+c+4abc\\
 &>  c+4abc\,.
 \end{align*}
 
 Now we turn to the second inequality. By collecting all non-square root terms on the left hand side and taking squares on both sides of the inequality we get the inequality
 $$4(ab+1)(ac+1)(bc+1)\leq (2abc+3c-a-b)^2.$$
 After expanding this inequality we are left to prove that
 $$8b^2ac+8a^2bc+4ab+4ac+4bc+4\leq 8c^2ab+(3c-a-b)^2$$
 holds. By Lemma \ref{lem:Jones}, we have $c\geq a+b+2r$. Thus, we have to check that the inequality
 $$4ab+4ac+4bc+4\leq 16rabc+(3c-a-b)^2$$
 holds. However, it is easy to see that
 $$4ab+4ac+4bc+4\leq 4abc+4abc+4abc+4abc\leq 16rabc+(3c-a-b)^2.$$
\end{proof}

The following two recent results are essential in the proof of Theorem \ref{thm:main} and will be used frequently without any special reference. In particular
the inequality $b>3a$ will be used several times.

\begin{lemma}[Theorem 1.1 of \cite{Cipu-Filipin-Fujita}, see also \cite{Cipu-Fujita}]\label{lem:b3a}
 Let $\{a,b,c,d,e\}$ be a Diophantine quintuple with $a < b < c < d < e$. Then $b > 3a$. Moreover, if $c > a+b+2\sqrt{ab+1}$ then $b > \max\{24 a, 2 a^{3/2}\}$.
\end{lemma}

\begin{lemma}[Theorem 1.5 of \cite{Fujita-Miyazaki}]\label{lem:cb}
 Let $\{a, b, c, d\}$ be a Diophantine quadruple with $a < b < c < d$. If $b < 2a$ and $c \geq 9.864 b^4$ or $2a \leq b \leq 12a$ and $c \geq 4.321b^4$ , or $b > 12a$ and
$c \geq 721.8b^4$ , then $d = d_+$. 
\end{lemma}

Combining these two results, we obtain the following result.

\begin{lemma}\label{lem:acb}
  Let $\{a,b,c,d,e\}$ be a Diophantine quintuple with $a < b < c < d < e$. Then, we have $ac<180.45 b^3$.  
\end{lemma}

\begin{proof}
Assume that $\{a, b, c, d, e\}$ is a Diophantine quintuple with $a < b < c < d < e$. By Fujita's result \cite{Fujita-1} (see Theorem \ref{thm:fujita}),
the Diophantine quadruple $\{a, b, c, d\}$ is regular. Consider the irregular Diophantine quadruple $\{a, b, d, e\}$. From Lemma~\ref{lem:b3a}, we have $b>3a$.
Hence, Lemma~\ref{lem:cb} provides $d<721.8b^4$. On the other hand, we have $d=d_{+} >4abc$.
Thus, we obtain $4abc<721.8b^4$ from which we conclude that $ac<180.45 b^3$.  
\end{proof}

In several concrete computations lower bounds for $b$, $c$ and $d$ are needed. Several authors have excluded many possibilities
for a pair $(a, b)$ to be extendable to a Diophantine quintuple. Using recent results due to Cipu and Fujita \cite{Cipu-Fujita} and Filipin et al. \cite{FFT2}, we can show the following.

\begin{lemma}\label{lem:min_bcd}
 If $\{a, b, c, d, e\}$ is a Diophantine quintuple, with $a<b<c<d<e$, then $b\geq 15$, $c\geq 24$ and $d\geq 1520$.
\end{lemma}

\begin{proof}
 Due to a result of Cipu and Fujita \cite{Cipu-Fujita} (cf. Lemma \ref{lem:b3a}) and a result due to Filipin et al. \cite{FFT2} we may consider only Diophantine quintuples
 such that $b>3a$ and $(a,b)\neq (k,4k\pm 4)$.
 
 First, note that if $r>15$ then $b>15$. Therefore, a simple computer search involving all Diophantine pairs $(a, b)$ not excluded by \cite{Cipu-Fujita} and \cite{FFT2} such that $2\leq r \leq 15$ shows that the Diophantine pair $\{1,15\}$ yields the minimal $b$. Since $a+b+2r$ strictly increases with $a$, we deduce that the smallest not excluded $c$ is $1+15+2\sqrt{1\cdot15+1}=24$. Therefore, $\{1,15,24\}$ is the smallest Diophantine triple which is possibly extendable to a Diophantine quintuple. Since any Diophantine triple $\{a, b, c\}$ which is extendable to a Diophantine quintuple $\{a, b, c, d, e\}$ with $a<b<c<d<e$ satisfies $d=d_+(a, b, c)$ (cf. Theorem \ref{thm:fujita}), we deduce that $d\geq d_+(1,15,24)=1520$.
\end{proof}

\section{An operator on Diophantine triples}\label{sec:3}

In this section we give a classification of Diophantine triples by defining the $\partial$-operator between Diophantine triples.
We start with a useful classification for Euler triples in terms of $d_{-}$.

\begin{proposition}\label{pro:2-1}
The Diophantine triple $\{a, b, c\}$ is an Euler triple if and only if $d_{-}(a, b, c) = 0$.
\end{proposition}

\begin{proof}
 If $\{a, b, c\}$ is an Euler triple, then  $c=a+b+2r$, with $r=\sqrt{ab+1}$. Moreover, we have 
$$
\sqrt{ac+1}=a+r ,\quad \sqrt{bc+1}= b+r. 
$$
Using these identities, we get
\begin{align*}
d_{-}(a, b, c) &= a+b+c + 2abc -  2\sqrt{(ab+1)(ac+1)(bc+1)}\\
&= 2a+2b+2r + 2ab(a+b+2r) - 2r(a+r)(b+r)\\
&= 2a+2b+2r + 2a^2b+2ab^2+4abr - 2r(2ab+ar+br+1)\\
&= 2a+2b+2r + 2a^2b+2ab^2 - 2a(ab+1) - 2b(ab+1)-2r=0.
\end{align*}

On the other hand, assuming that $d_{-}(a,b, c)=0$ implies
$$
(a+b+c + 2abc)^2 = 4(ab+1)(ac+1)(bc+1).
$$
Expanding and simplifying this equation, we get
$$
a^2+b^2+c^2-2ab-2ac-2bc=4.
$$ 
Some further manipulations yield
$$
(c-(a+b))^2 = 4(ab+1) = 4r^2. 
$$
Thus, we get
$$
c=a+b\pm 2r. 
$$
Since $c>b>a$ we may omit the "$-$" case, and the triple $\{a, b, c\}$ is indeed an Euler triple.
\end{proof}

The following formulas gathered in the next proposition will be also useful. Loosely speaking these formulas
show that $d_{-}$ and $d_{+}$ are in some sense the inverse functions of each other.

\begin{proposition}\label{pro:2-2}
Let $\{a, b, c\}$ be a Diophantine triple with $c=\max\{a, b, c\}$. We have 
$$
a=d_{-}(b,c,d_{+}(a,b,c)),\quad b=d_{-}(a,c,d_{+}(a,b,c)),\quad c=d_{-}(a,b,d_{+}(a,b,c)).
$$
Moreover, if $\{a,b,c\}$ is not an Euler triple, then we have
$$
c=d_{+}(a,b,d_{-}(a,b,c)).
$$
In particular, $\{a,b,d_{-1}(a,b,c),c\}$ is a regular Diophantine quadruple.
\end{proposition}

\begin{proof}
We fix $a$ and $b$ and consider $d_+(a,b,x):\RR^+\rightarrow \RR^+$ as a function. 
Let us consider the following equation with the unknown $x$ and with fixed $y\in \RR^+$ such that $y>\max\{a,b,x\}$:
$$a+b+x + 2abx +  2\sqrt{(ab+1)(ax+1)(bx+1)}=y.$$
This equation yields the quadratic equation 
$$(a+b+x + 2abx-y)^2=4(ab+1)(ax+1)(bx+1) $$
 and solving for $x$ yields the two solutions
$$x=a+b+y + 2aby + 2\sqrt{(ab+1)(ay+1)(by+1)}$$
and
$$x=a+b+y + 2aby - 2\sqrt{(ab+1)(ay+1)(by+1)}.$$
Obviously, the first solution can be discarded since we assume that $y>x$, while the second solution yields
$$x=d_{-}(a,b,y)=d_{-}(a,b,d_+(a,b,x)).$$
Since the formulas for $d_+$ and $d_-$ are symmetric in $a$, $b$ and $c$, we obtain the first three formulas.
To obtain the fourth formula we may use a similar trick. The last statement is a direct consequence from the fourth 
formula. Note that we have $c=d_+=d_+(a,b,d_-(a,b,c))$.
\end{proof}

For our next step the following observation will be useful:

\begin{lemma}\label{lem:well-def}
 Let $\{a,b,c\}$ be a Diophantine triple, then $d_{-}(a,b,c)\neq a,b,c$.
\end{lemma}

\begin{proof}
 Without loss of generality we assume that $a<b<c$.

 We start with the claim that $2rst>2abc+a+b$, which immediately shows that $d_{-}(a,b,c)<c$, hence $d_{-}(a,b,c)\neq c$. Indeed to verify our claim we have
 to show that
 \begin{align*}
 4(ab+1)(ac+1)(bc+1)&=4a^2b^2c^2+4a^2bc+4ab^2c+4abc^2\\
 &\;\;\;+4ab+4ac+4bc+4\\
 &>\left(2abc+a+b\right)^2\\
 &=4a^2b^2c^2+4a^2bc+4ab^2c+a^2+b^2+2ab.
 \end{align*}
 Hence we have to show that
 $$4abc^2+2ab+4ac+4bc+4>a^2+b^2,$$
 which can be easily seen.
 
 Therefore, let us consider the case that $d_{-}(a,b,c)=a$. Note that $d_{-}(a,b,c)=a$ 
 can be seen as a quadratic equation in $c$. Indeed the equation is equivalent to $b+c+2abc=2\sqrt{(ab+1)(ac+1)(bc+1)}$.
 Squaring both sides and solving for $c$ yields
 $$c=a+b\pm \sqrt{4ab+4+4a}.$$
 Since we assume that $c>a,b$ we may assume that $c=a+b+ \sqrt{4ab+4+4a}>a+b+2r$. Thus, by Lemma \ref{lem:Jones}, we have that
 $$c=a+b+\sqrt{4ab+4+4a}>4ab$$
 and by some further manipulations we obtain
 \begin{equation}\label{eq:well-def-ieq}
 4ab+4a+4>(4ab-a-b)^2=16a^2b^2-8a^2b-8ab^2+a^2+b^2+2ab>a^2+b^2+2ab.
 \end{equation}
By Lemma \ref{lem:b3a} we know that $b>3a$ and therefore we get
 $$4ab+4a+4>a^2+3ab+2ab$$
 hence $4a+4>4a^2$ and therefore $a=1$. But if we plug in $a=1$ into Inequality~\eqref{eq:well-def-ieq}
 we obtain $4b+4>1+b^2+2b$ which yields a contradiction unless $b<4$, but due to Lemma \ref{lem:min_bcd} we may assume that $b\geq 15$.
 
 A similar argument shows that also $b=d_{-}(a,b,c)$ is impossible.
\end{proof}

To any Diophantine triple $\{a,b,c\}$, we may add $d_{+}$ to obtain a regular Diophantine quadruple $\{a,b,c,d_{+}\}$. In particular,
we obtain from the triple $\{a,b,c\}$ three new Diophantine triples $\{a,b,d_{+}\}$, $\{a,c,d_{+}\}$ and $\{b,c,d_{+}\}$ related to  $\{a,b,c\}$.
From a naive point of view we may consider the triples $\{a,b,d_{+}\}$, $\{a,c,d_{+}\}$ and $\{b,c,d_{+}\}$ to be farther away from being
an Euler triple than the original triple $\{a,b,c\}$. Now, let us reverse these observations. Thus, given a non-Euler triple $\{a,b,c\}$ 
we want to get a new Diophantine triple $\{a',b',c'\}$ that is closer to the property of being an Euler triple. In order to specify these ideas
we introduce the $\partial$-operator:

\begin{definition}
We define $\partial$ to be an operator which sends a non-Euler triple $\{a,b,c\}$ to 
a Diophantine triple $\{a',b',c'\}$ such that
$$
\partial(\{a,b,c\}) = \{a,b,c,d_{-}(a,b,c)\} - \{\max(a,b,c)\},
$$
where $\{a,b,c,d_{-}(a,b,c)\} - \{\max(a,b,c)\}$ denotes the set which we obtain by removing the maximal element from the set  $\{a,b,c,d_{-}(a,b,c)\}$.

For a nonnegative integer $D$, we can define the operator $\partial_{-D}$ on Diophantine triples recursively by
\begin{enumerate}
\item For any Diophantine triple $\{a,b,c\}$ we define
$$
\partial _{0}(\{a,b,c\}) =\{a,b,c\}. 
$$

\item Provided that $\partial_{-(D-1)}(\{a,b,c\})$ is not an Euler triple, we recursively define
$$
\partial_{-D}(\{a,b,c\}) =\partial\left(\partial_{-(D-1)}(\{a,b,c\})\right),\quad \mbox{for}\,\, D\ge 1. 
$$ 
\end{enumerate}

Moreover, we put
$$
d_{-D}(a,b,c)=d_{-}(\partial_{-D+1}(\{a,b,c\})).
$$
\end{definition} 

In particular, we have that $\partial=\partial_{-1}$ and
$$
\partial_{-2}(\{a,b,c\}) =\partial\left(\partial_{-1}(\{a,b,c\})\right).
$$
Furthermore, we note that due to Lemma \ref{lem:well-def} the $\partial$-operator is well defined, i.e. a Diophantine triple $\{a,b,c\}$ is mapped indeed to another Diophantine triple,
unless $\{a,b,c\}$ is not an Euler triple.

\begin{proposition}\label{pro:2-3}
For any fixed Diophantine triple $\{a,b,c\}$, there exists a unique nonnegative integer $D<\frac{\log (abc)}{ \log 12}$ such that $d_{-(D+1)}(a,b,c)=0$. 
\end{proposition}

\begin{proof}
If $\{a,b,c\}$ is an Euler triple, then the result is obtained by Proposition \ref{pro:2-1}. 

Now, let us assume that $\{a,b,c\}$ is not an Euler triple. Since by Proposition \ref{pro:2-2} we know that $\{a,b,d_{-1}(a,b,c),c\}$
is a regular Diophantine quadruple, we deduce that $c>4ab\cdot d_{-1}(a,b,c)$ by Lemma \ref{lem:d+ieq}. In particular, we have 
$ab\cdot d_{-1}(a,b,c)<\frac{c}{4}<\frac{abc}{12}$. Note that $ab\geq 3$. This implies that the product $a'b'c'$ of the elements of the corresponding
triple $\{a',b',c'\}:=\partial_{-k}(\{a,b,c\})$ is less than $\frac{abc}{12^{k}}$, provided that the previous $k-1$ images were
not Euler triples. Thus, there exists some suitable positive integer $D$ less than $\frac{\log(abc)}{\log 12}$ such that
$\{a'',b'',c''\}:=\partial_{-D}(\{a,b,c\})$ is an Euler triple and by Proposition~\ref{pro:2-1}, we have $d_{-(D+1)}(a,b,c)=0$.

The uniqueness of $D$ is a direct consequence of the fact that the product $a'b'c'$ with $\{a',b',c'\}:=\partial_{-k}(\{a,b,c\})$ is strictly decreasing with $k$ until we get an Euler triple.
\end{proof}

\begin{definition}
We say that a Diophantine triple $\{a,b,c\}$ is of {\textit{degree}} $D$ and is {\textit{generated}} by an Euler triple
$\{a',b',c'\}$, if $d_{-(D+1)}(a,b,c)=0$ and $\partial_{-D}(\{a,b,c\})=\{a',b',c'\}$. If the triple $\{a,b,c\}$ is of degree $D$ we simply write $\deg(a,b,c)=D$. 
\end{definition} 

\begin{remark}
 Note that in the definition the triple $\{a',b',c'\}$ is an Euler triple due to Proposition \ref{pro:2-1} since $d_-(a',b',c')=0$ by assumption.
\end{remark}

\begin{remark} 
Note that there are not too many Diophantine triples of small degree which are generated by a fixed Euler triple $\{a,b,c\}$.
Indeed, for an arbitrary but fixed Euler triple $\{a,b,c\}$, there are at most $3^{D}$ Diophantine triples generated by $\{a,b,c\}$ with $\deg(a,b,c)= D$. 
\end{remark}

\begin{proof}
The proof will be done by induction on $D$. If $D=0$, then $\{a,b,c\}$ is the only Diophantine triple of degree $0$ generated by $\{a,b,c\}$.
Assume that the result holds
for $D=k\ge 0$. Every triple $\{a',b',c'\}$ which is generated by $\{a,b,c\}$ and which is of degree $D$ yields three new triples 
$\{b',c',d_{+}(a',b',c')\}$, $\{a',c',d_{+}(a',b',c')\}$ and $\{a',b',d_{+}(a',b',c')\}$, which are all of degree $D+1$ and which are sent back 
to $\{a',b',c'\}$ by the $\partial$ operator, due to Proposition~\ref{pro:2-2}. Thus, we obtain at most $3\cdot 3^D=3^{D+1}$ new triples of degree $D+1$.
\end{proof}

Now, we know the structure of all Diophantine triples which are generated by some Euler triples. 
For each Euler triple, we get a ternary tree with root $\{a,b,a+b+2r\}$. For example, the triples
$\{1,3,120\}$, $\{1,8,120\}$, $\{3,8,120\}$ have degree $1$ and are generated by the Fermat triple $\{1,3,8\}$.

\section{System of Pell equations}\label{sec:4}

Let $\{a, b, c\}$ be a Diophantine triple with $a < b < c$, and $r$, $s$, $t$ positive integers such that
$$
ab + 1 = r^2, \quad ac + 1 = s^2,\quad bc + 1 = t^2. 
$$
Furthermore, suppose that $\{a, b, c, d, e\}$ is a Diophantine quintuple with $a < b < c < d < e$, and put
$$ad + 1 = x^2,\quad  bd + 1 = y^2,\quad cd + 1 = z^2,$$
with positive integers $x, y, z$. Then, there exist integers $X,Y,Z,W$ such that
$$
ae + 1 = X^2,\quad be + 1 = Y^2,\quad ce + 1 = Z^2,\quad de + 1 = W^2. 
$$
Note that if we fix $d=d_{+}$, which we may assume due to Fujita's result~\cite{Fujita-1} (cf. Theorem \ref{thm:fujita}), then we have
$$
x= at+rs,\quad y= bs +rt, \quad z=cr+st.
$$
By eliminating $e$ from the above equations, we obtain the following system of Pell equations:
\begin{align}\label{eq:ab}
 aY^2 - bX^2 &= a - b,\\\label{eq:ac}
 aZ^2 - cX^2 &= a - c,\\\label{eq:bc}
 bZ^2 - cY^2 &= b - c,\\\label{eq:ad}
 aW^2 - dX^2 &= a - d,\\\label{eq:bd}
 bW^2 - dY^2 &= b - d,\\\label{eq:cd}
 cW^2 - dZ^2 &= c - d. 
\end{align}
Let us state the following result concerning Pell equations of the form \eqref{eq:ab}--\eqref{eq:cd}.

\begin{lemma}\label{lem:boundsx_0y_0}
 Every integer solution to a Pell equation of the form
 $$ aY^2-bX^2=a-b,$$
 with $ab+1=r^2$ is obtained from
 $$Y \sqrt{a} + X \sqrt{b} = (y_0 \sqrt{a} + x_0 \sqrt{b})(r+\sqrt{ab})^{n},$$
 where $n$, $x_0$ and $y_0$ are integers such that $n\geq 0$,
 $$1\leq x_0 \leq \sqrt{\frac{a(b-a)}{2(r-1)}},\quad \text{and} \quad 1\leq |y_0| \leq \sqrt{\frac{(r-1)(b-a)}{2a}}.$$
\end{lemma}

\begin{proof}
 The Lemma is a direct application of the general theory of Pell equations as described in \cite[Theorem 108a]{Nagell}
 after some small modifications. In particular, this specific case  has also been studied by Dujella~\cite[Lemma 1]{Dujella:2001}.
\end{proof}

We apply Lemma \ref{lem:boundsx_0y_0} to the system of Pell equations \eqref{eq:ab}--\eqref{eq:cd} and obtain
\begin{align}\label{eq:abi}
Y \sqrt{a} + X \sqrt{b} &=Y^{(a,b)}_{h'} \sqrt{a} + X^{(a,b)}_{h'} \sqrt{b} = (Y_0 \sqrt{a} + X_0 \sqrt{b})(r+\sqrt{ab})^{h'},\\\label{eq:acj}
Z \sqrt{a} + X \sqrt{c} &=Z^{(a,c)}_{j'} \sqrt{a} + X^{(a,c)}_{j'} \sqrt{c} = (Z_1 \sqrt{a} + X_1 \sqrt{c})(s+\sqrt{ac})^{j'},\\\label{eq:bck}
Z \sqrt{b} + Y \sqrt{c} &=Z^{(b,c)}_{k'} \sqrt{b} + Y^{(b,c)}_{k'} \sqrt{c} = (Z_2 \sqrt{b} + Y_2 \sqrt{c})(t+\sqrt{bc})^{k'},\\\label{eq:adl}
W \sqrt{a} + X \sqrt{d} &=W^{(a,d)}_{l'} \sqrt{a} + X^{(a,d)}_{l'} \sqrt{d} = (W_3 \sqrt{a} + X_3 \sqrt{d})(x+\sqrt{ad})^{l'},\\\label{eq:bdm}
W \sqrt{b} + Y \sqrt{d} &=W^{(b,d)}_{m'} \sqrt{b} + Y^{(b,d)}_{m'} \sqrt{d} = (W_4 \sqrt{b} + Y_4 \sqrt{d})(y+\sqrt{bd})^{m'},\\\label{eq:cdn}
W \sqrt{c} + Z \sqrt{d} &=W^{(c,d)}_{n'} \sqrt{c} + Z^{(c,d)}_{n'} \sqrt{d} = (W_5 \sqrt{c} + Z_5 \sqrt{d})(z+\sqrt{cd})^{n'},
\end{align}
for some nonnegative integers $h',j',k',l',m',n'$ and integers $Y_0$, $X_0$, $Z_1$, $X_1$, $Z_2$, $Y_2$, $W_3$, $X_3$, $W_4$, $Y_4$, $W_5$, and $Z_5$.

In view of the relations \eqref{eq:adl}, \eqref{eq:bdm}, and \eqref{eq:cdn} we have that 
 $W=W^{(a,d)}_{l'} = W^{(b,d)}_{m'}=W^{(c,d)}_{n'}$, where $W^{(a,d)}_{l'}$, $W^{(b,d)}_{m'}$, and $W^{(c,d)}_{n'}$ satisfy the following recursions:
\begin{align*}
W^{(a,d)}_0 &= W_3,& W^{(a,d)}_1 &= xW_3 + d X_3,& W^{(a,d)}_{l'+2} &= 2x W^{(a,d)}_{l'+1} -W^{(a,d)}_{l'}, \\
W^{(b,d)}_0 &= W_4,& W^{(b,d)}_1 &= yW_4 + d Y_4,& W^{(b,d)}_{m'+2} &= 2y W^{(b,d)}_{m'+1} - W^{(b,d)}_{m'},\\ 
W^{(c,d)}_0 &= W_5,& W^{(c,d)}_1 &= zW_5 + d Z_5,& W^{(c,d)}_{n'+2} &= 2z W^{(c,d)}_{n'+1} - W^{(c,d)}_{n'}.
\end{align*}

The next two lemmas will help us to better understand the structure of the solutions to the system of Pell equations
\eqref{eq:ab}--\eqref{eq:cd}. In particular, the next lemma due to Fujita~\cite{Fujita-number} takes care
of the subsystem of Pell equations \eqref{eq:ad}--\eqref{eq:cd}

\begin{lemma}[\cite{Fujita-number}, Lemma~2.2]\label{lem:1}
If $W=W^{(a,d)}_{l'} = W^{(b,d)}_{m'}=W^{(c,d)}_{n'}$, then we have
$$
l'\equiv m'\equiv n'\equiv 0 \pmod 2
$$
 and 
$$
W_3=W_4=W_5=\varepsilon =\pm 1.
$$
\end{lemma}

Our next aim is to extend Fujita's result stated in Lemma \ref{lem:1}.

\begin{lemma}\label{lem:2}
We have $$h'\equiv j'\equiv k' \equiv 0 \pmod{2}$$and $$X_0=X_1=Y_0=Y_2=Z_1=Z_2=1.$$ 
\end{lemma}

\begin{proof}
First we consider the subsystem of Pell equations \eqref{eq:ab} and \eqref{eq:bd}, i.e. the system
\begin{equation}\label{eqs:abd}
\begin{split}
aY^2 - bX^2 &= a - b, \\
bW^2 - dY^2 &= b - d.
\end{split}
\end{equation}
In particular we are interested in the first equation of system \eqref{eqs:abd}.
By Lemma \ref{lem:boundsx_0y_0}, the integer solutions to the first equation of \eqref{eqs:abd} are obtained by   
$$
Y \sqrt{a} + X \sqrt{b} = (Y_0 \sqrt{a}+ X_0 \sqrt{b})(r+\sqrt{ab})^{h'},\quad h'\ge 0,
$$
where  $|Y_0| \le \sqrt{\frac{(r-1)(b-a)}{2a}}<\sqrt{\frac{b\sqrt{b}}{2\sqrt{a}}}<0.71b^{3/4}$. 
Further, by \eqref{eq:abi} we may write $Y = Y^{(a,b)}_{h'}$, where $Y^{(a,b)}_{h'}$ satisfies the following recursion
$$
Y^{(a,b)}_{0} = Y_0, \quad Y^{(a,b)}_{1} = r Y_0  + b X_0, \quad Y^{(a,b)}_{h'+2} = 2r  Y^{(a,b)}_{h'+1} -  Y^{(a,b)}_{h'}. 
$$
Thus, we get 
\begin{equation}\label{eq:betaab}
 Y^{(a,b)}_{h'}\equiv  \begin{cases}
Y_0\phantom{r}\pmod{b},& \quad \mbox{if}\,\, h'\,\, \mbox{even}, \\
Y_0r\pmod{b},& \quad \mbox{if}\,\, h'\,\, \mbox{odd}.
                                      \end{cases}
\end{equation}

Next we consider the second equation of system \eqref{eqs:abd}.
By Lemma~\ref{lem:1}, we know that $W_4=\varepsilon =\pm1$. This implies together with Pell equation \eqref{eq:bd} that $Y_4=\pm 1$.
Note that due to Lemma \ref{lem:boundsx_0y_0}, we may discard the case that $Y_4=-1$ and therefore we may assume that $Y_4=1$.
Using \eqref{eq:bdm}, we have $Y=Y^{(b,d)}_{m'}$ and we obtain the following recursion:
$$
Y^{(b,d)}_0 = 1, \quad Y^{(b,d)}_1 = y +\varepsilon b, \quad Y^{(b,d)}_{m'+2} = 2y Y^{(b,d)}_{m'+1} - Y^{(b,d)}_{m'}. 
$$
From this recursion we deduce that
\begin{equation}\label{eq:betabd}
Y^{(b,d)}_{m'}\equiv 1 \pmod{b},\quad \mbox{for even}\,\,m'. 
\end{equation}
Note that due to Lemma \ref{lem:1} we may assume that $m'$ is even.
If $Y^{(a,b)}_{h'}=Y^{(b,d)}_{m'}$ with $h'$ odd, then from the congruences \eqref{eq:betaab} and \eqref{eq:betabd} 
we have that $ Y_0 r\equiv 1 \pmod{b}$. 
Multiplying both sides by $r$, we obtain $Y_0 \equiv r \pmod{b} $. The bound $|Y_0|<0.71b^{3/4}$ implies $Y_0 = r$ or $Y_0=r-b$.
Let us consider the case $Y_0=r$ first.
As $(Y_0, X_0)$ is an integer solution to the first equation in \eqref{eqs:abd} we obtain
$$
X_0^2 =\frac{aY_0^2 -a+b}b = \frac{a(ab+1)-a+b}b = a^2+1. 
$$
But $X_0^2=a^2+1$ implies that $X_0=1$ and $a=0$, which is a contradiction.

Now, assume that $Y_0=r-b$. Since $b>3a$ due to Lemma \ref{lem:b3a} we have $a\leq b/3 -1$, hence
\begin{equation}\label{eq:Y_0sol}
0.71b^{3/4}>b-r\geq b-\sqrt{(b/3 -1)b+1}\geq b\left(1-\frac1{\sqrt 3}\right)>0.422 b.
\end{equation}
Therefore, we obtain $b\leq 8.1<15$, which is a contradiction to Lemma \ref{lem:min_bcd}.

Hence, if $Y=Y^{(a,b)}_{h'}=Y^{(b,d)}_{m'}$, then $h'$ is even. Once again, from the congruences~\eqref{eq:betaab} and~\eqref{eq:betabd},
we obtain that $ Y_0 \equiv 1 \pmod{b}$. By $|Y_0|<0.71b^{3/4}$, we have $Y_0=1$. Thus, we get $X_0=\pm 1$.
By Lemma \ref{lem:boundsx_0y_0}, we may assume that $X_0$ is positive and we obtain $X_0=Y_0=1$. 

Similarly, by replacing $Y$ by $Z$ and $b$ by $c$ we obtain from the system of Pell equations \eqref{eqs:abd} the system 
\begin{equation*}
 \begin{split}
aZ^2 - cX^2 &= a - c, \\
cW^2 - dZ^2 &= c - d.
\end{split}
\end{equation*}
Now, applying the same arguments as above, we deduce that $j'$ is even and $X_1=Z_1=1$. In particular, in this case we have to exclude the solution $Z_1=s-c$.
However in this case, instead of inequality \eqref{eq:Y_0sol} we obtain the inequality
$$0.71c^{3/4}>|Z_1|=c-s\geq c-\sqrt{(c/3 -1)c+1}\geq c\left(1-\frac1{\sqrt 3}\right)>0.422 c$$
and deduce that $j'$ is even.

The same method also works if we replace $Y$ by $Z$, $X$ by $Y$, $a$ by $b$ and $b$ by $c$ in the system of Pell equations \eqref{eqs:abd}, i.e if we consider the system
\begin{equation*}
\begin{split}
bZ^2 - cY^2 &= b - c, \\
cW^2 - dZ^2 &= c - d.
\end{split}
\end{equation*}
In this case, we achieve that $k'$ is even and $Y_2=Z_2=1$. Thus, the only non straightforward step is to exclude the fundamental solution $Z_2=t-c$. If $\{a,b,c\}$ is not an Euler triple, then we have $c\geq 4ab$ and in particular $b\leq c/4$. Therefore,  instead of inequality \eqref{eq:Y_0sol} we obtain the inequality
$$0.71c^{3/4}>|Z_2|=c-t\geq c-\sqrt{\frac{c^2}4+1}=c\left(1-\sqrt{\frac14+\frac{1}{c^2}}\right)>0.498 c,$$
which yields a contradiction to the fact that $c\geq 24$. Also note that the last inequality is due to $c\geq 24$. Therefore, let us assume that $c=a+b+2r$. Then, we obtain
\begin{multline*}
s=c-t=|Z_2|\leq \sqrt{\frac{(t-1)(c-b)}{2b}} \\
<\sqrt{\frac{\sqrt{bc}(a+b+2r-b)}{2b}}<\sqrt{\sqrt{c/b}(a/2+r)}<\sqrt{s\cdot s}=s,
\end{multline*}
which is a contradiction. Hence, in any case we obtain that $k'$ is even.
\end{proof}

Furthermore, due to Lemma \ref{lem:boundsx_0y_0} we deduce that $X_3=Y_4=Z_5=1$ since $W_3=W_4=W_5=\varepsilon=\pm 1$.  So from now and on, we may write
$$
h'=2h,\;\; j'=2j,\;\; k'=2k,\;\; l'=2l,\;\; m'=2m, \;\; n'=2n,
$$
where $h,j,k,l,m,n$ are positive integers. Note that we may assume that the exponents are positive because a vanishing exponent would yield that either $X$, $Y$ or $Z$ is one, thus one of $a,b,c,d,e$ is zero.
Finally, we may rewrite the formulas \eqref{eq:abi}-\eqref{eq:cdn} to
\begin{align}\label{eq:abii}
Y \sqrt{a} + X \sqrt{b} &= (\sqrt{a} + \sqrt{b})(r+\sqrt{ab})^{2h},\\\label{eq:acjj}
Z \sqrt{a} + X \sqrt{c} &= (\sqrt{a} + \sqrt{c})(s+\sqrt{ac})^{2j},\\\label{eq:bckk}
Z \sqrt{b} + Y \sqrt{c} &= (\sqrt{b} + \sqrt{c})(t+\sqrt{bc})^{2k},\\\label{eq:adll}
W \sqrt{a} + X \sqrt{d} &= (\varepsilon \sqrt{a} + \sqrt{d})(x+\sqrt{ad})^{2l},\\\label{eq:bdmm}
W \sqrt{b} + Y \sqrt{d} &= (\varepsilon \sqrt{b} + \sqrt{d})(y+\sqrt{bd})^{2m},\\\label{eq:cdnn}
W \sqrt{c} + Z \sqrt{d} &= (\varepsilon \sqrt{c} + \sqrt{d})(z+\sqrt{cd})^{2n}.
\end{align}

\section{The gap principle and the classical congruence}\label{sec:5}

If there exists a positive integer $e$ such that the Diophantine quadruple $\{a,b,c,d\}$ can be extended
to a quintuple $\{a,b,c,d,e\}$ with $a<b<c<d<e$, then relations \eqref{eq:abii}--\eqref{eq:cdnn}
are fulfilled with positive integers $(i,j,k,l,m,n)$.  Moreover, there are $12$ sequences associated with $X,Y,Z$, and $W$, 
with indices $i,j,k,l,m$, and $n$ as defined in \eqref{eq:abi}--\eqref{eq:cdn}. For example, we have 
$$Y^{(a,b)}_{h} \sqrt{a} + X^{(a,b)}_h \sqrt{b}: = (\sqrt{a} + \sqrt{b})(r+\sqrt{ab})^{h},$$ 
and so on. Thus, each of the variables $X,Y,Z$, and $W$ corresponds to three sequences as follows:
\begin{align*}
X&=X^{(a,b)}_{2h} = X^{(a,c)}_{2j} = X^{(a,d)}_{2l},& Y&=Y^{(a,b)}_{2h} = Y^{(b,c)}_{2k} = Y^{(b,d)}_{2m}, \\
Z&=Z^{(a,c)}_{2j} = Z^{(b,c)}_{2k} = Z^{(c,d)}_{2n},& W&=W^{(a,d)}_{2l} = W^{(b,d)}_{2m} = W^{(c,d)}_{2n}.
\end{align*}
Several authors proved various relations between the indices of these sequences. Let us recall two results
due to Dujella \cite{Dujella:2004} and Fujita \cite{Fujita-number} respectively.

\begin{lemma}[Lemma 3 of \cite{Dujella:2004}]\label{lem:jk}
If $Z=Z^{(a,c)}_{2j} = Z^{(b,c)}_{2k}$, then $k-1\le j \le 2k+1$. 
\end{lemma} 

\begin{lemma}[Lemma 2.3 of \cite{Fujita-number}]\label{lem:lmn}
If $W=W^{(a,d)}_{2l} = W^{(b,d)}_{2m} = W^{(c,d)}_{2n}$, then $4\le n\le m\le l \le  2n$.
\end{lemma}

Lemmas \ref{lem:jk} and \ref{lem:lmn} reveal relations between the indices $k,j,m$, and $l$. In order to
get a better understanding of the relations between $m,n$, and $h$ we prove the following two lemmas.

\begin{lemma}\label{lem:2l3m}
We have $2l\le 3 m$ and $m<l$ unless $m=0$.
\end{lemma}

\begin{proof}
From \eqref{eq:adll} and \eqref{eq:bdmm} respectively \eqref{eq:adl} and \eqref{eq:bdm}, we get the recursions
\begin{align*}
W^{(a,d)}_0&=\varepsilon,& W^{(a,d)}_1 &= \varepsilon x+ d,& W^{(a,d)}_{l+2}& =2xW^{(a,d)}_{l+1}-W^{(a,d)}_{l},& l &\ge 0, \\
W^{(b,d)}_0&=\varepsilon,& W^{(b,d)}_1 &= \varepsilon y+ d,& W^{(b,d)}_{m+2}& =2yW^{(b,d)}_{m+1}-W^{(b,d)}_{m},& m &\ge 0.
\end{align*}
Furthermore, solving these recursions explicitly we get
\begin{align*}
W_{2l}^{(a,d)}&=\frac{d+\varepsilon\sqrt{ad}}{2\sqrt{ad}}\left(x+\sqrt{ad}\right)^{2l}+\frac{-d+\varepsilon\sqrt{ad}}{2\sqrt{ad}}\left(x-\sqrt{ad}\right)^{2l}\\
W_{2m}^{(b,d)}&=\frac{d+\varepsilon\sqrt{bd}}{2\sqrt{bd}}\left(y+\sqrt{bd}\right)^{2m}+\frac{-d+\varepsilon\sqrt{bd}}{2\sqrt{bd}}\left(y-\sqrt{bd}\right)^{2m}.
\end{align*}

We prove that $2l\le 3 m$ first. By considering the intersection of the two recursions $W^{(a,d)}_{2l}$ and $W^{(b,d)}_{2m}$, i.e. the equation $W^{(a,d)}_{2l} = W^{(b,d)}_{2m}$, we get the inequality
\begin{multline*}
\frac{d+\varepsilon\sqrt{ad}}{2\sqrt{ad}}\left(x+\sqrt{ad}\right)^{2l}-\frac{d+\sqrt{ad}}{2\sqrt{ad}}<W^{(a,d)}_{2l} = W^{(b,d)}_{2m}\\
< \frac{d+\varepsilon\sqrt{bd}}{2\sqrt{bd}}\left(y+\sqrt{bd}\right)^{2m}.
\end{multline*}
The inequality above holds since $0<x-\sqrt{ad},\; y-\sqrt{bd}<1$. We add $\frac{d+\sqrt{ad}}{2\sqrt{ad}}$ to both sides of the inequality and multiply by $\frac{2\sqrt{ad}}{d+\varepsilon\sqrt{ad}}$ afterwards. Then, we get
\begin{align*}
(x+\sqrt{ad})^{2l}&<\frac{\sqrt{a}}{\sqrt{b}}\cdot\frac{d+\varepsilon\sqrt{bd}}{d+\varepsilon\sqrt{ad}}(y+\sqrt{bd})^{2m}+\frac{d+\sqrt{ad}}{d-\sqrt{ad})}\\
&<\left(\frac{\sqrt{a}}{\sqrt{b}}\cdot \frac{d+3\sqrt{bd}}{d-\sqrt{ad}}+1\right)(y+\sqrt{bd})^{2m}.
\end{align*}
Note that
$$\frac{d+3\sqrt{bd}}{d-\sqrt{ad}}\leq 1+\frac{4\sqrt{bd}}{d-\sqrt{ad}}=1+\frac{4\sqrt{b/d}}{1-\sqrt{a/d}}< 1.42$$
since $d>4abc$ and $b\geq 15$ and $c\geq 24$ due to Lemmas \ref{lem:d+ieq} and \ref{lem:min_bcd}. Moreover we have that $b>3a$, i.e. $\sqrt{a/b}\leq \sqrt{1/3}$. Therefore, we obtain that 
$$\frac{\sqrt{a}}{\sqrt{b}}\cdot \frac{d+3\sqrt{bd}}{d-\sqrt{ad}}+1<1.42\cdot \sqrt{1/3}+1<1.82.$$

Assume that $2l\ge 3m+1$. Then, we have that $(x+\sqrt{ad})^{3m+1}<1.82(y+\sqrt{bd})^{2m}$. Since $x+\sqrt{ad}>1.82$, we get $(x+\sqrt{ad})^{3m}<(y+\sqrt{bd})^{2m}$ and furthermore 
$$
(x+\sqrt{ad})^{3}<(y+\sqrt{bd})^{2}.
$$ 
From $x+\sqrt{ad}>2\sqrt{ad}$ and $y+\sqrt{bd}<2.1\sqrt{bd}$ (note that $b\geq 15$ and $d\geq 1520$), we have 
$$
64a^3d^3=(4ad)^3 <(4.41bd)^2<20b^2d^2
$$
which yields 
$$
3.2a^3d<b^2. 
$$
But this is a contradiction to the fact that $d=d_+ > 4abc >4ab^2.$ Therefore, we have $2l < 3m+1$ and deduce $2l\le 3m$. 

The proof that $m<l$ is similar. However, let us prove that $l=m=1$ is not a solution to $W^{(a,d)}_{2l} = W^{(b,d)}_{2m}$ separately. That is
we consider the equation
$$2xd+2\varepsilon x^2-\varepsilon=2yd+2\varepsilon y^2-\varepsilon.$$
But since $0<x<y<d$ and since the function $f(x)=2xd+2\varepsilon x^2-\varepsilon$ is strictly increasing for $x>-d$ if $\varepsilon=1$ and strictly decreasing for $x<d$ if $\varepsilon=-1$
the above equation cannot hold. Since we know by Lemma \ref{lem:lmn} that $m\leq l$, we may assume for the rest of the proof that $2m=2l\geq 4$.

Thus, we consider the equation $W^{(a,d)}_{2m} = W^{(b,d)}_{2m}$ and obtain
the inequality 
\begin{multline*}
\frac{d+\varepsilon\sqrt{ad}}{2\sqrt{ad}}\left(x+\sqrt{ad}\right)^{2m}>W^{(a,d)}_{2m} = W^{(b,d)}_{2m}\\
> \frac{d+\varepsilon\sqrt{bd}}{2\sqrt{bd}}\left(y+\sqrt{bd}\right)^{2m}-\frac{d+\sqrt{bd}}{2\sqrt{bd}}.
\end{multline*}
By some similar manipulations as above we obtain
$$
 \left(y+\sqrt{bd}\right)^{2m}<\left(\sqrt{\frac ba}\frac{d+\sqrt{ad}}{d-\sqrt{bd}}+\frac{d+\sqrt{ad}}{d-\sqrt{bd}}\right)(x+\sqrt{ad})^{2m}.
$$
Since
\begin{align*}
\sqrt{\frac ba}\frac{d+\sqrt{ad}}{d-\sqrt{bd}}&+\frac{d+\sqrt{ad}}{d-\sqrt{bd}}=\sqrt{\frac ba}\left( 1+\frac{\sqrt{ad}+\sqrt{bd}}{d-\sqrt{bd}}+\sqrt{\frac ab}+\frac{2\sqrt{ad}}{d-\sqrt{bd}}\right)\\
&<\sqrt{\frac ba}\left(1+\sqrt{1/3}+\frac{3\sqrt{ad}+\sqrt{bd}}{d-\sqrt{bd}}\right)\\
&<\sqrt{\frac ba}\left(1+\sqrt{1/3}+\frac{3\sqrt{\frac{1}{4bc}}+\sqrt{\frac{1}{4ac}}}{1-\sqrt{\frac{1}{4ac}}}\right)\\
&<1.78 \sqrt{\frac ba}
\end{align*}
we get
\begin{equation}\label{eq:m<l_ieq}
\left(\frac{y+\sqrt{bd}}{x+\sqrt{ad}}\right)^{2m}<1.78 \sqrt{\frac ba}.
\end{equation}
On the other hand we have 
$$
\frac{y+\sqrt{bd}}{x+\sqrt{ad}}>\frac{2\sqrt{bd}}{\sqrt{ad+1}+\sqrt{ad}}>0.999\sqrt{\frac ba}.
$$
Note that $d\geq 1520$ by Lemma \ref{lem:min_bcd} and $\frac ba>3$ by Lemma \ref{lem:b3a}. Since we may assume from the discussion above that $2m\geq 4$ we obtain from inequality \eqref{eq:m<l_ieq}
that
$$
0.996\frac {b^2}{a^2}<\left(\frac{y+\sqrt{bd}}{x+\sqrt{ad}}\right)^{2m}<1.78 \sqrt{\frac ba}.
$$
which yields $b/a<1.48$ which is a contradiction to Lemma \ref{lem:b3a} which states that $b/a>3$.
Therefore inequality \eqref{eq:m<l_ieq} holds only if $m=l=0$, i.e. we have that $m<l$ unless $m=0$.
\end{proof}

\begin{lemma}\label{lem:hm}
We have $h\ge 2m$.
\end{lemma}

\begin{proof}
We proceed similarly as in the proof of Lemma \ref{lem:2l3m}. In this case, we consider the equation $Y=Y^{(a,b)}_{2h}  = Y^{(b,d)}_{2m}$. From \eqref{eq:abii} and \eqref{eq:bdmm}, we obtain the recursions
\begin{align*}
Y^{(a,b)}_0&=1,& Y^{(a,b)}_1& = r + b,& Y^{(a,b)}_{h+2} &=2rY^{(a,b)}_{h+1}-Y^{(a,b)}_{h},& h&\ge 0, \\
Y^{(b,d)}_0&=1,& Y^{(b,d)}_1& =  y+ \varepsilon b,& Y^{(b,d)}_{m+2} &=2yY^{(b,d)}_{m+1}-Y^{(b,d)}_{m},& m&\ge 0.
\end{align*}

Solving the first recursion, we get
$$
Y_{2h}^{(a,b)}=\frac{b+\sqrt{ab}}{2\sqrt{ab}}(r+\sqrt{ab})^{2h}-\frac{b-\sqrt{ab}}{2\sqrt{ab}}(r-\sqrt{ab})^{2h}
$$
and deduce that
$$Y^{(a,b)}_{2h}<\frac{1}{2}\left(1+\sqrt{\frac{b}{a}}\right)(r+\sqrt{ab})^{2h}.$$

Solving the second recursion, we obtain
$$
Y_{2m}^{(b,d)}=\frac{\varepsilon b+\sqrt{bd}}{2\sqrt{bd}}(y+\sqrt{bd})^{2m}+\frac{\sqrt{bd}-\varepsilon b}{2\sqrt{bd}}(y-\sqrt{bd})^{2m},
$$
which yields 
$$
Y^{(b,d)}_{2m}>\frac{\sqrt{bd}-b}{2\sqrt{bd}}(y+\sqrt{bd})^{2m}>\frac{1}{2}(1-\sqrt{b/d})(y+\sqrt{bd})^{2m}.
$$

If $Y^{(a,b)}_{2h}  = Y^{(b,d)}_{2m}$, then we have
$$
(y+\sqrt{bd})^{2m}<\frac{1+\sqrt{\frac{b}{a}}}{1-\sqrt{\frac{b}{d}}}(r+\sqrt{ab})^{2h}<(r+\sqrt{ab})^{2h+1}
$$
since
$$
\frac{1+\sqrt{\frac{b}{a}}}{1-\sqrt{\frac{b}{d}}}\leq \left(\sqrt{\frac ba}+1\right)\frac{1}{1-\sqrt{\frac 1{4ac}}}<1.114\sqrt{ab}(1+\frac{1}{\sqrt{ba}}) <1.41 \sqrt{ba} < r+\sqrt{ab}.
$$
Note that $ac\geq 24$ and $ab \geq 15$ in any case.

We claim that $y+\sqrt{bd}>(r+\sqrt{ab})^2$ and deduce from this claim that 
$$
(r+\sqrt{ab})^{4m}<(r+\sqrt{ab})^{2h+1},
$$
which shows that $4m<2h+1$. Thus we get $4m\le 2h$, hence $h\ge 2m$.

Therefore, we are left to justify our claim. In order to show the claim, it suffices to prove that
$$y+\sqrt{bd}>2\sqrt{bd}>\left(2\sqrt{\frac 43 ab}\right)^2>\left(2\sqrt{ab+1}\right)^2>\left(r+\sqrt{ab}\right)^2.$$
The only non obvious inequality is the second one. Squaring both sides yields $bd>\frac{64}9 a^2b^2$.
Since $d>4abc$, we have to show that $ac>\frac{16}9 a^2$, which is true since $c>b>3a$ (cf. Lemma \ref{lem:b3a}).
\end{proof}

Next, let us state the following useful observation.

\begin{lemma}\label{lem:cong_eq}
 Let $\{a,b,c,d,e\}$ be a Diophantine quintuple, then we have
\begin{equation}\label{eq:ab4d}
al^2 + \varepsilon xl \equiv bm^2 +\varepsilon ym\pmod{4d}.
\end{equation}
\end{lemma}

\begin{proof}
 The congruence above is a direct consequence of a result due to Dujella \cite[Lemma 4]{Dujella:2001}
 in which several congruence relations between the indices where shown. In order to obtain the result,
 we apply \cite[Lemma 4]{Dujella:2001} to the Diophantine quadruple $\{a,b,d,e\}$ and note that since $\{a,b,d,e\}$
 is part of a Diophantine quintuple we may consider only the even cases of \cite[Lemma 4]{Dujella:2001} (i.e. part (1) and (3) of \cite[Lemma 4]{Dujella:2001})
 and due to Lemmas \ref{lem:1} and \ref{lem:2} we obtain congruence~\eqref{eq:ab4d}.
\end{proof}

Most researchers studying Diophantine quintuples used similar congruences to discuss lower bounds for various indices.
In \cite{Wu-He:2014}, Wu and the first author got a strong lower bound for $m$, namely $m\ge 0.48 \sqrt{d/b}$. 
A slight improvement of the constant from $0.48$ to $0.5$ was achieved by Cipu~\cite{Cipu}. An even better lower bound was given by Cipu and Trudgian in~\cite{Cipu-Trudgian}.
However we surpass these bounds by proving the following lemma.   

\begin{lemma} \label{lem:mdb}
If $W^{(a,d)}_{2l} = W^{(b,d)}_{2m}$, then $m\ge \left(\frac{\sqrt{17}-1}{2}\right)\sqrt{d/b}$.  
\end{lemma}

\begin{proof}
We consider congruence relation \eqref{eq:ab4d} and assume for the moment that 
\begin{equation}\label{eq:abyx}
al^2 + \varepsilon xl = bm^2 +\varepsilon ym,
\end{equation}
then we obtain
$$
al^2-bm^2 = \varepsilon (ym-xl).
$$
This implies that 
\begin{align*}
(ym+xl)\left(al^2-bm^2\right) &=\varepsilon\left(y^2m^2 - x^2l^2\right)\\
& = \varepsilon\left((bd+1)m^2 - (ad+1)l^2\right)\\
 & = \varepsilon\left(d(bm^2 - al^2) + m^2 -l^2\right).
\end{align*}
Collecting terms and taking absolute values results in
$$
\left|l^2 - m^2\right| = \left|(d+\varepsilon(ym+xl))(bm^2 - al^2)\right|.
$$
If $bm^2 - al^2 =0$ or $d+\varepsilon(ym+xl)=0$, then $l=m$. Due to Lemma \ref{lem:2l3m} we deduce that $l=m=0$, which is impossible. Hence,
$$
\left|l^2-m^2\right| \ge \left|bm^2 - al^2\right|
$$
and we obtain the inequality
$$
\left|\frac{b}{a} - \frac{l^2}{m^2}\right| \le \frac{\left|l^2 - m^2\right|}{am^2} = \frac{\left|(l/m)^2-1\right|}{a}.
$$
Using Lemma~\ref{lem:2l3m}, we have $(l/m)^2<2.25$ and by Lemma~\ref{lem:b3a} we get $b/a>3$. Thus, we obtain
\begin{equation}\label{eq:ineq}
0.75=3-2.25\le \left|\frac{b}{a} - \frac{l^2}{m^2}\right| \le \frac{|(l/m)^2-1|}{a}<\frac{1.25}{a},
\end{equation}
i.e. $a<\frac{5}{3}$. Therefore, we only need to consider the case that $a=1$. When $a=1$ and $b\geq 15$ (cf. Lemma \ref{lem:min_bcd}), then inequality \eqref{eq:ineq} 
is impossible. Therefore, relation \eqref{eq:abyx} does not hold. 

In the case that equation \eqref{eq:abyx} does not hold, the left side and the right side of \eqref{eq:abyx} differ at least by $4d$. Therefore, we get the inequality
$$4d\leq |bm^2-al^2+\varepsilon (ym-xl)| \leq |bm^2-al^2| + |ym-xl| < bm^2 +ym. $$
Thus,  we have 
\begin{equation}\label{ieq:lem12}
4d\le bm^2 +ym-1 =bm^2 + m\sqrt{bd} +\frac{m}{\sqrt{bd+1}+\sqrt{bd}}-1,
\end{equation}
since
$$ym-m\sqrt{bd}  = m(\sqrt{bd+1}-\sqrt{bd})  = \frac{m}{\sqrt{bd+1}+\sqrt{bd}}. $$
Assume for the moment that $m\le\frac{\sqrt{17}-1}{2}\sqrt{d/b}$, then we have
$$
m\le\frac{\sqrt{17}-1}{2}\sqrt{d/b}<\frac{\sqrt{17}-1}{2}\sqrt{d/15}<\sqrt{d}<\sqrt{bd+1}+\sqrt{bd}. 
$$
Therefore, we have $\frac{m}{\sqrt{bd+1}+\sqrt{bd}}-1<0$. With these inequalities at hand,
we obtain from inequality \eqref{ieq:lem12} the following inequality
$$
4d<bm^2+m\sqrt{bd}\frac{(\sqrt{17}-1)^2}{4}d + \frac{\sqrt{17}-1}{2}d =4d,
$$
which is impossible. Hence, we must have $m> \frac{\sqrt{17}-1}{2} \sqrt{d/b}$. 
\end{proof}

Combining all the above lemmas yields the main result of this section.

\begin{lemma}\label{lem:hbd}
We have $h>(2\sqrt{17}-2)\sqrt{ac}>6.2462\sqrt{ac}$. 
\end{lemma}

\begin{proof}
Combining Lemma~\ref{lem:d+ieq}, Lemma~\ref{lem:hm} and Lemma~\ref{lem:mdb}, we immediately get
$$
h\ge 2m \ge (\sqrt{17}-1)\sqrt{d/b} > (\sqrt{17}-1)\sqrt{4abc/b} 
$$ 
$$
=(2\sqrt{17}-2)\sqrt{ac}>6.2462\sqrt{ac}.
$$
\end{proof}

\section{Linear forms in logarithms}\label{sec:6}

In the last section, we found a lower bound for the exponent $h$. This section is devoted to finding good
upper bounds for $h$. This is done by using lower bounds for linear forms in logarithms. To formulate 
the results concerning lower bounds for linear forms in logarithms, we recall the notation of
logarithmic height.

For any non-zero algebraic number $\gamma$ of degree $D$
over $\mathbb{Q}$, whose minimal polynomial over $\mathbb{Z}$ is $A\prod_{j=1}^D
\left(X-\gamma^{( j)} \right)$, we denote by
$$
h(\gamma) = \frac{1}{D} \left( \log A + \sum_{j=1}^D
\log\max\left(1, \betrag{\gamma\conj j}\right)\right)
$$
its absolute logarithmic height. With this notation at hand, we can state the following useful
result due to Matveev \cite{Matveev:2000}.

\begin{theorem}\label{thm:Matveev}
Let $\Lambda$ be a linear form in logarithms of $N$ multiplicatively independent totally real algebraic
numbers $\alpha_1,\ldots, \alpha_N$ with rational integer coefficients $b_1, \ldots, b_N$ such that $b_N  \neq 0$.
Let $h(\alpha_j)$ denote the absolute logarithmic height of $\alpha_j$, for $1\le j \le N$.
Define the numbers $D,A_j\; (1\le j \le N)$ and $E$ by 
$D:=[\mathbb{Q}(\alpha_1,\ldots, \alpha_{N}):\QQ]$, $A_j := \max\{ Dh(\alpha_j), |\log \alpha_j|\}$ and $E:=\max\{1, \max\{|b_j| A_j/A_N; 1\le j \le N\}\}$. Then,   
$$
\log |\Lambda| > - C(N) C_0 W_0 D^2 \Omega,  
$$
where 
\begin{gather*}
C(N): = \frac{8}{(N-1)!} (N+2) (2N+3)(4e(N+1))^{N+1},\\  
C_0 := \log (e^{4.4N+7} N^{5.5} D^2 \log(eD)),\\
W_0 := \log (1.5eED\log(eD)),\qquad  \Omega := A_1\cdots A_N. 
\end{gather*}
\end{theorem}

The main focus of this section will lie on the intersection $X=X^{(a,b)}_{2h} = X^{(a,c)}_{2j}$. From \eqref{eq:abii} and \eqref{eq:acjj},
we obtain recursions for $X^{(a,b)}_{2h}$ and $X^{(a,c)}_{2j}$. By solving these recursions, we obtain
\begin{equation} \label{eq:Xab}
\begin{split}
X^{(a,b)}_{2h}& = \frac{(\sqrt{a} + \sqrt{b})(r+\sqrt{ab})^{2h}-(\sqrt{a} - \sqrt{b})(r-\sqrt{ab})^{2h}}{2\sqrt{b}}\\
X^{(a,c)}_{2j}& = \frac{ ( \sqrt{a} + \sqrt{c})(s+\sqrt{ac})^{2j}-  ( \sqrt{a} - \sqrt{c})(s-\sqrt{ac})^{2j}}{2\sqrt{c}}.
\end{split}
\end{equation}
This motivates us to define
\begin{equation}\label{eq:Lamb1}
\Lambda_1 = 2h \log(r+\sqrt{ab})- 2j \log(s+\sqrt{ac}) + \log \left(\frac{\sqrt{c}(\sqrt{a}+\sqrt{b})}{\sqrt{b}(\sqrt{a}+\sqrt{c})}\right).
\end{equation} 
Our first aim is to show that $\Lambda_1$ is a rather small, but positive number.

\begin{lemma}\label{lem:low-Lamb1}
$0<\Lambda_1<(s+\sqrt{ac})^{-4j}$. 
\end{lemma}

\begin{proof}
We follow the ideas of Baker and Davenport \cite{Baker-Davenport:1969} (see also \cite[Lemma 5]{Dujella:2001}).  Let 
$$
P= \frac{(\sqrt{a} + \sqrt{b})(r+\sqrt{ab})^{2h}}{\sqrt{b}},\quad Q=  \frac{ ( \sqrt{a} + \sqrt{c})(s+\sqrt{ac})^{2j}}{\sqrt{c}}.
$$
Using the explicit formulas \eqref{eq:Xab} and substituting $P$ and $Q$ in the right way, we can rewrite the equation $X^{(a,b)}_{2h} = X^{(a,c)}_{2j}$ to 
$$
P+\frac{b-a}{b}P^{-1} = Q+\frac{c-a}{c}Q^{-1}. 
$$
This yields 
$$
P-Q=\frac{c-a}{c}Q^{-1}- \frac{b-a}{b}P^{-1} >\frac{c-a}{c}(Q^{-1}- P^{-1})=\frac{c-a}c\frac{P-Q}{PQ}.
$$
In the case that $P-Q<0$, the above inequality would yield $1<\frac{c-a}c \frac{1}{PQ}=(1-\frac{a}c) \frac{1}{PQ}<1$, which is
an obvious contradiction. Therefore, we have $P-Q>0$, hence $\Lambda_1 = \log\frac{P}{Q}>0$. 

On the other hand, we have 
$$
P-Q<\frac{c-a}{c}Q^{-1}<Q^{-1}
$$
and we obtain 
$$
\Lambda_1=\log \frac{P}{Q} < \log(1+Q^{-2})<Q^{-2}<(s+\sqrt{ac})^{-4j}. 
$$
\end{proof}

Now, we apply Theorem~\ref{thm:Matveev} to $\Lambda_1$ with 
\begin{gather*}
N=3,\quad D=4,\quad  b_1=2h,\quad b_2=-2j,\quad b_3=1,\\
\alpha_1= r+\sqrt{ab},\quad  \alpha_2 = s+\sqrt{ac},\quad \alpha_3=\frac{\sqrt{c}(\sqrt{a}+\sqrt{b})}{\sqrt{b}(\sqrt{a}+\sqrt{c})}.
\end{gather*}

Notice that $X^2-2rX+1=0$ is the minimal polynomial of $\alpha_1=r+\sqrt{ab}$ and
that $X^2-2sX+1=0$ is the minimal polynomial $\alpha_2=s+\sqrt{ac}$. Therefore, we get 
$$
 h(\alpha_1)=\frac{1}{2}\log\alpha_1,\quad h(\alpha_2)=\frac{1}{2}\log \alpha_2.
$$

Since the absolute values of the conjugates of $\alpha_3$ which are $\geq 1$ are 
$$
\frac{\sqrt{c}(\sqrt{a}+\sqrt{b})}{\sqrt{b}(\sqrt{a}+\sqrt{c})},\quad \frac{\sqrt{c}(\sqrt{a}+\sqrt{b})}{\sqrt{b}(-\sqrt{a}+\sqrt{c})},
$$
and since the minimal polynomial of $\alpha_3$ is
$$
b^2(c-a)^2X^4-4b^2c(c-a)X^3+2bc(3bc-a^2-ab-ac)X^2-4bc^2(b-a)X+c^2(b-a)^2,
$$
we obtain that 
$$
h(\alpha_3)=\frac{1}{4}\log\left(b^2(c-a)^2\cdot \frac{c}{b} \cdot \frac{(\sqrt{a}+\sqrt{b})^2}{c-a}\right)<\log c.
$$
Thus, we choose
$$
A_1 =2\log \alpha_1,\;\;A_2=2\log\alpha_2,\;\; A_3=4\log c.
$$

Next, we compute the quantity $E$. By the definition of $\Lambda_1$ and Lemma~\ref{lem:low-Lamb1}, we have $|b_3|A_3<|b_1|A_1<|b_2|A_2$.
Indeed since $\Lambda_1>0$, we deduce that $|b_1|A_1<|b_2|A_2$. As $4\log c<2j \log (2\sqrt{ac})$ and $j\geq 1$, we have $|b_3|A_3<|b_1|A_1$. Therefore, we get
$$
E=\frac{|b_2|A_2}{A_3}=\frac{j\log\alpha_2}{\log c}<\frac{h \log \alpha_3+\frac 12\log \alpha_1}{\log c}\leq h.
$$
The last inequality can be seen by showing that $\log\alpha_1+\frac 12\log\alpha_3\leq \log c$. Indeed, we have
\begin{align*}
(\sqrt{ab+1}&+\sqrt{ab}) \sqrt{\frac{\sqrt{c}(\sqrt{a}+\sqrt{b})}{\sqrt{b}(\sqrt{a}+\sqrt{c})}} \\
&= (\sqrt{ab+1}+\sqrt{ab}) \sqrt{1+\frac{\sqrt{ac}-\sqrt{ab}}{\sqrt{bc}+\sqrt{ab}}}\\
&\leq  (\sqrt{ab+1}+\sqrt{ab}) \sqrt{1+\sqrt{\frac ab}}\\
&\leq  (\sqrt{ab+1}+\sqrt{ab}) \left(1+\frac 12 \sqrt{\frac ab}\right) \qquad (\text{Bernoulli's inequality})\\
&=  \frac{a}2 + \frac{a}2 \sqrt{1+\frac{1}{ab}}+\sqrt{ab+1}+\sqrt{ab}\\
&\leq  \frac a2+\frac a2\left(1+ \frac{1}{2ab}\right)+\sqrt{ab+1}+\sqrt{ab}\qquad \qquad  \;\;\; (\text{Bernoulli's inequality})\\     
&\leq  a+b+2\sqrt{ab+1}\leq c.
\end{align*}

Before we may apply Theorem~\ref{thm:Matveev}, we also have to ensure that $\alpha_1$, $\alpha_2$ and $\alpha_3$ are multiplicatively independent.

\begin{lemma}\label{lem:mult_ind}
  With the notations above, the algebraic numbers $\alpha_1$, $\alpha_2$ and $\alpha_3$ are multiplicatively independent.
\end{lemma}

\begin{proof}
 First, we note that $\alpha_1$ and $\alpha_2$ are units in the fields $\QQ(\sqrt{ab})$ and $\QQ(\sqrt{ac})$, respectively.
 Since $a^2bc$ is not a perfect square, these two fields are two distinct extensions of $\QQ$ hence $\alpha_1$ and $\alpha_2$ are multiplicatively independent.
 Furthermore, computing the norm of $\alpha_3$, we obtain
 $$\mathrm{N}_{\mathbb{K}/\QQ}(\alpha_3)=\frac{c^2(b-a)^2}{b^2(c-a)^2}\neq \pm 1,$$
 where $\mathbb{K}=\QQ(\sqrt{ab},\sqrt{ac})$.
 Hence $\alpha_3$ is not a unit and therefore $\alpha_1$, $\alpha_2$ and $\alpha_3$ are indeed multiplicatively independent.
\end{proof}

Now by an application of Theorem~\ref{thm:Matveev}, we have
   \begin{equation}\label{eq:Matveev-lowerbound}
    \log|\Lambda|> -  4.928\cdot 10^{12}\cdot \log\left(38.92h\right)\cdot  \log\alpha_1 \cdot \log\alpha_2 \cdot \log c.
\end{equation}
Combining inequality \eqref{eq:Matveev-lowerbound} with Lemma~\ref{lem:low-Lamb1}, we obtain that
$$
4h\log\alpha_1<4j\log\alpha_2< 4.928\cdot 10^{12}\cdot \log\left(38.92h\right)\cdot  \log\alpha_1 \cdot \log\alpha_2 \cdot \log c.
$$
Therefore, we obtain the inequality  
$$
\frac{h}{\log(38.92 h)}< 1.232\cdot 10^{12}\cdot \log\alpha_2 \cdot \log c.
$$
Since $\alpha_2=s+\sqrt{ac+1}<2\sqrt{ac+1}$, we get
\begin{equation}\label{eq:nd}
\frac{h}{\log(38.92h)}< 1.232\cdot 10^{12}\cdot \log (2\sqrt{ac+1})\cdot \log c.
\end{equation}
By Lemma~\ref{lem:hbd}, we have $h>6.2462\sqrt{ac}$. Moreover, $\frac{h}{\log(38.92h)}$ is an increasing function if $h\geq 1$ and we deduce that  
\begin{equation}
\sqrt{ac}<1.98\cdot 10^{11}\cdot \log (2\sqrt{ac+1})\cdot \log c\cdot\log (243.11\sqrt{ac}).
\end{equation}
A straightforward computation gives $ac<6.18\cdot 10^{32}$. Inserting this into \eqref{eq:nd} we get
$$\frac{h}{\log(38.92h)}< 3.577\cdot 10^{15},$$
which yields $h<1.55\cdot 10^{17}$. 
Moreover, we deduce that $d<4abc+4c<4(ac)^2+4ac<1.53\cdot 10^{66}$. Summarizing these results yields

\begin{proposition}\label{pro:1}
  Suppose that $\{a,b,c,d,e\}$ is a Diophantine quintuple with $a < b < c < d < e$.
  Then we have $ac<6.18\cdot 10^{32}$, $d<1.53\cdot 10^{66}$ and $h<1.55\cdot 10^{17}$.  
\end{proposition}  

This upper bound for $d$ (and also for $h$) is a slight improvement of that obtained by Wu and the first author~\cite{Wu-He:2014}
and it is a little weaker than the bounds obtained in some special cases obtained by Cipu and Trudgian~\cite{Cipu-Trudgian}.

In order to get a sharper bound, we use this bound together with a powerful tool due to Mignotte \cite{Mignotte:kit}.
In fact, some slightly different versions  of the following theorem were used in some papers, cf. Theorem 12.9 of \cite{BMS1},  Theorem 3 of \cite{BMS2}. 
We use the statement of Proposition 5.1 of \cite{Mignotte:kit}. One can refer to the results in Section 12 of \cite{BMS1} and get 
there the details of the proof. A slightly modified version is Proposition 3.3 in \cite{Bennett-Pinter:2006}. 

\begin{theorem}\label{thm:kit}
We consider three non-zero algebraic numbers $\alpha_1, \alpha_2$ and $\alpha_3$, which are either all real and $> 1$ or all
complex of modulus one and all $\neq 1$. Moreover, we assume that either the three numbers $\alpha_1, \alpha_2$ and $\alpha_3$
are multiplicatively independent, or two of these numbers are multiplicatively independent and the third one is a root of unity. Put
$$
\mathcal{D} = [\QQ(\alpha_1, \alpha_2,\alpha_3): \QQ]/[\RR(\alpha_1, \alpha_2,\alpha_3): \RR].
$$

We also consider three positive coprime rational integers $b_1, b_2, b_3$, and the linear form
$$
\Lambda=b_2\log\alpha_2 - b_1\log\alpha_1  - b_3\log\alpha_3,
$$
where the logarithms of the $\alpha_i$ are arbitrary determinations of the logarithm, but which are all real or all purely imaginary.

And we assume also that
$$
b_2|\log\alpha_2| =b_1|\log\alpha_1|  + b_3|\log\alpha_3| \pm |\Lambda|. 
$$
We put
$$
d_1=\gcd(b_1,b_2), \quad d_3=\gcd(b_3,b_2),\quad b_1=d_1b_1',\quad b_2=d_1b_2'=d_3b_2'',\quad b_3=d_3b_3''. 
$$
Let $\rho\ge \exp(1)$ be a real number. Put $\lambda=\log \rho$. Let $a_1$, $a_2$ and $a_3$ be real numbers such that 
$$
a_i\ge \rho|\log \alpha_i| - \log|\alpha_i| + 2\mathcal{D}h(\alpha_i), \quad i=1,2,3,
$$
and assume further that 
$$
\Omega:=a_1a_2a_3\ge 2.5 \quad \mbox{and} \quad A:=\min\{a_1,a_2,a_3\}\ge 0.62.
$$
Let $K$, $L$, and $M$ be positive integers with 
$$
L\ge 4+ \mathcal{D}, \quad K= \lfloor M\Omega L \rfloor, \quad \mbox{where} \,\, M\ge 3. 
$$
Let $\chi>0$ be fixed and $\le 2$. Define 
\begin{align*}
c_1& = \max\left\{(\chi M L)^{2/3}, \sqrt{2ML/A}  \right\},\\
c_2& = \max\left\{2^{1/3}(M L)^{2/3}, \sqrt{M/A}L  \right\},\\
c_3& =(6M^2)^{1/3}L, 
\end{align*}
and then put 
\begin{align*}
R_1&=\lfloor c_1a_2a_3 \rfloor,& S_1&= \lfloor c_1a_1a_3 \rfloor,& T_1&=\lfloor c_1a_1a_2 \rfloor,\\
R_2&=\lfloor c_2a_2a_3 \rfloor,& S_2&= \lfloor c_2a_1a_3 \rfloor,& T_2&=\lfloor c_2a_1a_2 \rfloor,\\ 
R_3&=\lfloor c_3a_2a_3 \rfloor,& S_3&= \lfloor c_3a_1a_3 \rfloor,& T_3&=\lfloor c_3a_1a_2 \rfloor.
\end{align*}
Let also 
$$
R=R_1+R_2+R_3+1, \quad S=S_1+S_2+S_3+1, \quad  T=T_1+T_2+T_3+1.
$$
Define 
$$
c_0=\max\left\{\frac{R}{La_2a_3}, \frac{S}{La_1a_3},\frac{T}{La_1a_2}\right\}. 
$$
Finally, assume that 
\begin{equation}\label{eq:condition1}
\begin{split}
\left(\frac{KL}{2}\right. &\left. + \frac{L}{4}-1 - \frac{2K}{3L}\right)\lambda + 2 \mathcal{D} \log 1.36 \\
&\ge (\mathcal{D}+1) \log L + 3gL^2 c_0 \Omega +\mathcal{D}(K-1)\log\tilde{b}+2\log K,
\end{split}
\end{equation}
where 
$$
g=\frac{1}{4} - \frac{K^2L}{12RST},\quad b'=\left(\frac{b_1'}{a_2}+\frac{b_2'}{a_1}\right)\left(\frac{b_3''}{a_2}+\frac{b_2''}{a_3}\right),
\quad \tilde{b}= \frac{\exp(3)c_0^2\Omega^2 L^2}{4K^2} \times b'.
$$

Then \textbf{either} 
$$
\log |\Lambda| > - (KL + \log (3K L))\lambda, 
$$
-\textbf{or} \textbf{(A1)}: there exist two non-zero rational integers $r_0$ and $s_0$ such that 
$$
r_0b_2 = s_0 b_1
$$
with 
$$
|r_0|\le \frac{(R_1+1)(T_1+1)}{\mathcal{M}-T_1}\quad \mbox{and} \quad |s_0|\le \frac{(S_1+1)(T_1+1)}{\mathcal{M}-T_1},
$$
where 
\begin{align*}
\mathcal{M} &= \max\{R_1+S_1+1, S_1+T_1+1, R_1+T_1+1, \chi \mathcal{V}\},\\
\mathcal{V} &=\sqrt{(R_1+1)(S_1+1)(T_1+1)},
\end{align*}
\textbf{or} \textbf{(A2)}: there exist rational integers $r_1$, $s_1$, $t_1$, and $t_2$, with $r_1s_1\neq 0$ such that 
$$
(t_1b_1+r_1b_3)s_1 = r_1b_2t_2,\quad \gcd(r_1,t_1) = \gcd(s_1,t_2)=1,
$$
which also satisfy 
\begin{align*}
|r_1s_1| &\le \delta \cdot \frac{(R_1+1)(S_1+1)}{\mathcal{M}-\max\{R_1,S_1\}},\\
|s_1t_1| &\le \delta \cdot \frac{(S_1+1)(T_1+1)}{\mathcal{M}-\max\{S_1,T_1\}},\\ 
|r_1t_2| &\le \delta \cdot \frac{(R_1+1)(T_1+1)}{\mathcal{M}-\max\{R_1,T_1\}},
\end{align*}
where $\delta = \gcd(r_1,s_1)$. Moreover, when $t_1=0$ we can take $r_1=1$, and when $t_2=0$ we can take $s_1=1$. 
\end{theorem}

\begin{remark}
The cases \textbf{(A1)} and \textbf{(A1)} represent the case \textbf{(C3)} of Theorem 2 in \cite{Mignotte:kit}.
\end{remark}

We aim to apply Theorem~\ref{thm:kit} to 
$$
\Lambda:=-\Lambda_1 = 2j \log(s+\sqrt{ac}) -2h \log(r+\sqrt{ab})-  \log \left(\frac{\sqrt{c}(\sqrt{a}+\sqrt{b})}{\sqrt{b}(\sqrt{a}+\sqrt{c})}\right).
$$
Therefore, we go through the theorem step by step. First, let us assume for technical reasons that $c>2\cdot 10^{8}$. As in the previous case we
take the parameters
\begin{gather*}
 \mathcal{D}=4,\quad  b_1=2h,\quad b_2=2j,\quad b_3=1,\\
\alpha_1= r+\sqrt{ab},\quad  \alpha_2 = s+\sqrt{ac}, \quad \alpha_3= \frac{\sqrt{c}(\sqrt{a}+\sqrt{b})}{\sqrt{b}(\sqrt{a}+\sqrt{c})} .
\end{gather*}
As already shown during the proof of Proposition \ref{pro:1}, we have
$$
h(\alpha_1) = \frac{1}{2}\log \alpha_1, \quad h(\alpha_2) = \frac{1}{2}\log \alpha_2, \quad  h(\alpha_3)<\log c.
$$ 
Moreover, let us note that
\begin{align*}
\log \alpha_3 &= \log \left(1+\frac{\sqrt{a}(\sqrt{c}-\sqrt{b})}{\sqrt{b}(\sqrt{a}+\sqrt{c})}\right)\\
&<\log\left(1+\sqrt{\frac{a}{b}}\right)<\log\left(1+\sqrt{\frac 13}\right)<0.46.  
\end{align*}
In view of this inequality, we set
$$
a_1 = (\rho+3)\log \alpha_1, \quad a_2 =(\rho+3)\log \alpha_2,\quad  a_3= 0.46(\rho-1)+8\log c.
$$
We make the following choices for our parameters:
$$\chi=2,\quad L=625,\quad M=12.1, \quad \rho=10.$$
These choices together with the assumption that $c\geq 10^6$ imply
$$A=\min\{a_1,a_2,a_3\}>13\log(\sqrt c)>89.8$$
and therefore, we obtain
$$c_1=611.59452\dots,\quad c_2=485.42289\dots,\quad c_3=5985.77903\dots\, .$$
With these values, we are able to compute
\begin{align*}
R_1 &= \lfloor c_1a_2a_3\rfloor \simeq \phantom{00}63605.83059\log\alpha_2(\log c+0.5175),\\ 
R_2 &= \lfloor c_2a_2a_3\rfloor \simeq \phantom{00}50483.98119\log\alpha_2(\log c+0.5175),\\ 
R_3 &= \lfloor c_3a_2a_3\rfloor \simeq \phantom{0}622521.01991\log\alpha_2(\log c+0.5175),\\ 
S_1 &= \lfloor c_1a_1a_3\rfloor \simeq \phantom{00}63605.83059\log\alpha_1(\log c+0.5175),\\ 
S_1 &= \lfloor c_2a_1a_3\rfloor \simeq \phantom{00}50483.98119\log\alpha_1(\log c+0.5175),\\ 
S_1 &= \lfloor c_3a_1a_3\rfloor \simeq \phantom{0}622521.01991\log\alpha_1(\log c+0.5175),\\ 
T_1 &= \lfloor c_1a_1a_2\rfloor \simeq \phantom{0}103359.47470\log\alpha_1\log\alpha_2,\\ 
T_2 &= \lfloor c_2a_1a_2\rfloor \simeq \phantom{00}82036.46944\log\alpha_1\log\alpha_2,\\ 
T_3 &= \lfloor c_3a_1a_2\rfloor \simeq 1011596.65736\log\alpha_1\log\alpha_2,
\end{align*}
where $R_1\simeq 63605.83059\log\alpha_2(\log c+0.5175) $ should be read as
$$63605.83059\log\alpha_2(\log c+0.5175)-1<R_1<63605.8306\log\alpha_2(\log c+0.5175)$$
and so on.

Next, we want to find an upper bound for $c_0$. Therefore, we note that we have
\begin{align*}
\frac{R}{La_2a_3}&=\frac{R_1+R_2+R_3+1}{La_2a_3}\\
& \leq \frac{c_1 a_2a_3+c_2a_2a_3+c_3a_2a_3+1}{La_2a_3}\\
& <\frac{c_1+c_2+c_3+1}L<11.3341.
\end{align*}
Since similar estimates hold for $\frac{S}{La_1a_3}$ and $\frac{T}{La_1a_2}$, we obtain
$$c_0<11.3341.$$

With the above choices, we have
\begin{align*}
\Omega &= 1352\log \alpha_1\log \alpha_2 (\log c +0.5175),\\
K &\simeq 10224500\log \alpha_1\log \alpha_2 (\log c +0.5175),
\end{align*}
where we interpret $K \simeq 10224500\log \alpha_1\log \alpha_2 (\log c +0.5175)$ as above.

Our next task is to show that inequality \eqref{eq:condition1} is satisfied. Therefore, we split up the inequality into four parts.
One part is representing the left hand side of inequality \eqref{eq:condition1} and the other three parts represent $(\mathcal{D}+1) \log L+ 2\log K$,
$3gL^2c_0\Omega$ and $\mathcal{D}(K-1)\log \tilde{b}$ respectively.
\begin{enumerate}[(i)]
 \item As $M\Omega L-1<K\le M\Omega L$, we have 
 \begin{multline*}
\left(\frac{KL}{2}+ \frac{L}{4}-1 - \frac{2K}{3L}\right)\lambda + 2 \mathcal{D} \log 1.36 \\
>7.357094\cdot 10^{9} \log \alpha_1\log \alpha_2 \log c +3.807296\cdot 10^9\log \alpha_1\log \alpha_2 -359.7. 
\end{multline*}

\item Using the upper bound for $K$, we obtain
$$(\mathcal{D}+1) \log L+ 2\log K< 87.73.$$

\item Using the explicit formula for $\Omega$ and noting that $g>\frac 14$, we get
$$3gL^2c_0\Omega < \frac{3}{4}L^2c_0\Omega < 4.4894\cdot 10^9 \log \alpha_1\log \alpha_2 \log c +2.3233\cdot 10^9 \log \alpha_1\log \alpha_2.$$

\item For the last part we start by estimating $b'$. First, let us note that since $\frac{j}{\log \alpha_1}<\frac{h+1}{\log \alpha_2}$, 
we have $\frac{b_2}{a_1}<\frac{2h+2}{a_2}$. Moreover, we have that $2\log\alpha_2>\log c$, hence  $\frac{b_3}{a_2}<\frac{2}{13 \log c}<\frac{2}{a_3}$.
Also note that $j\leq h$. Finally, let us note that by Proposition~\ref{pro:1}, we have that $h \le 1.55\cdot 10^{17}$. Therefore, we get
$$
 b'\le \left(\frac{b_1}{a_2}+\frac{b_2}{a_1}\right)\left(\frac{b_3}{a_2}+\frac{b_2}{a_3}\right)
 <\frac{(4h+2)(2h+2)}{104\log\alpha_2\log c}< 6.324\cdot 10^{29}.  
$$
Thus, we get 
$$
\log \tilde{b} < \log\left(\frac{6.324\cdot 10^{29}\, e^3\, c_0^2\Omega^2 L^2}{4K^2}\right)<70.1024,
$$
which establishes the fourth and last part of inequality \eqref{eq:condition1}: 
\begin{align*}
\mathcal{D}(K-1)\log \tilde{b}& <4M\Omega L \log \tilde{b}\\
& <2.8671\cdot 10^9\log \alpha_1 \log \alpha_2 \log c +1.4837\cdot 10^9\log \alpha_1 \log \alpha_2.
\end{align*}
\end{enumerate}
Combining (i)--(iv), we can now easily verify that condition \eqref{eq:condition1} is satisfied.

According to Theorem \ref{thm:kit}, we either obtain a lower bound for $|\log \Lambda_1|$ or one of the
additional cases \textbf{(A1)} and \textbf{(A2)} holds. First, let us consider the lower bound for $\log |\Lambda_1|$, which
is according to Theorem \ref{thm:kit} 
\begin{align*}
\log|-\Lambda_1|&>- (KL + \log (3K L))\lambda>-(ML^2 \Omega + \log (3ML^2 \Omega))\log \rho\\
&>-1.52656\cdot 10^{10} \log \alpha_1\log \alpha_2 \log c.
\end{align*}
On the other hand, Lemma~\ref{lem:low-Lamb1} implies that $\log|-\Lambda_1|<-4j\log \alpha_2$. Also note that
$h\log \alpha_1< j\log \alpha_2$, hence 
\begin{equation}\label{eq:hh}
h<3.8164\cdot 10^{9}\log \alpha_2 \log c.
\end{equation}

Before explicitly solving inequality \eqref{eq:hh}, we discuss the other two options \textbf{(A1)}
and \textbf{(A2)} of Theorem~\ref{thm:kit}. We start by computing $\mathcal{M}$. As we choose $\chi =2$, we get 
$$
\mathcal{M} = \chi \mathcal{V} = 2\mathcal{V}=2\sqrt{(R_1+1)(S_1+1)(T_1+1)}.
$$

First, we consider option $\textbf{(A2)}$ and compute
\begin{align*}
B_1 &:= \frac{(R_1+1)(S_1+1)}{\mathcal{M}-\max\{R_1,S_1\}},\\
B_2 &:= \frac{(S_1+1)(T_1+1)}{\mathcal{M}-\max\{S_1,T_1\}},\\
B_3 &:= \frac{(R_1+1)(T_1+1)}{\mathcal{M}-\max\{R_1,T_1\}}. 
\end{align*}
By our assumptions and Proposition~\ref{pro:1}, we have $2\cdot 10^{8}<ac<6.18\cdot 10^{32}$. Moreover, 
Lemma~\ref{lem:acb} implies that $b>c^{1/3}$. Thus, we get
$$
29.93 \le 13\log \alpha_1=a_1<a_2 =13\log\alpha_2,\qquad 114.66\le a_3 = 8(\log c+ 0.5175). 
$$
Finally, recall that
$$
c_1a_2a_3-1<R_1\le c_1a_2a_3,\quad c_1a_1a_3-1<S_1\le c_1a_1a_3,\quad c_1a_1a_2-1<T_1\le c_1a_1a_2,
$$
which imply that
$$
\mathcal{M}>2c_1^{3/2}a_1a_2a_3.
$$
Therefore, we obtain the following upper bounds:
\begin{align*}
B_1 &\le \frac{(c_1a_2a_3+1)(c_1a_1a_3+1)}{2c_1^{3/2}a_1a_2a_3-c_1a_2a_3}\\
& = \left(0.5c_1^{1/2}a_3+\frac{1}{2c_1^{1/2}a_1}\right)\cdot \frac{1+\frac{1}{c_1a_2a_3}}{1-\frac{1}{2c_1^{1/2}a_1}}\le 102.734\log c,\\
B_2 &\le \frac{(c_1a_1a_3+1)(c_1a_1a_2+1)}{2c_1^{3/2}a_1a_2a_3-c_1a_1a_3}\\
&= \left(0.5c_1^{1/2}a_1+\frac{1}{2c_1^{1/2}a_2}\right)\cdot \frac{1+\frac{1}{c_1a_1a_3}}{1-\frac{1}{2c_1^{1/2}a_2}}\\
& < \left(0.5c_1^{1/2}a_2+\frac{1}{2c_1^{1/2}a_2}\right)\cdot \frac{1+\frac{1}{c_1a_1a_3}}{1-\frac{1}{2c_1^{1/2}a_2}}\le 160.814\log\alpha_2,\\
B_3 &\le \frac{(c_1a_2a_3+1)(c_1a_1a_2+1)}{2c_1^{3/2}a_1a_2a_3-c_1a_2a_3}\\
& = \left(0.5c_1^{1/2}a_2+\frac{1}{2c_1^{1/2}a_1}\right)\cdot \frac{1+\frac{1}{c_1a_2a_3}}{1-\frac{1}{2c_1^{1/2}a_1}}\le 160.915\log\alpha_2.
\end{align*}
Since we assume that condition \textbf{(A2)} holds, there exist rational integers $r_1$, $s_1$, $t_1$, and $t_2$, with $r_1s_1\neq 0$ such that 
$$
(t_1b_1+r_1b_3)s_1 = r_1b_2t_2,\quad \gcd(r_1,t_1) = \gcd(s_1,t_2)=1,
$$
with
$$
|r_1s_1| \le \delta B_1, \quad |s_1t_1| \le \delta  B_2, \quad  |r_1t_2| \le \delta  B_3, \quad \delta = \gcd(r_1,s_1).$$
Put $r_1=\delta r_1'$ and $s_1 = \delta s_1'$. As $b_1= 2h$, $b_2 = 2j$, $b_3=1$ option \textbf{(A2)} is
$$
s_1't_1\cdot 2h + \delta r_1's_1'  = r_1' t_2 \cdot 2j,
$$
with
$$
|\delta r_1's_1'| \le  B_1, \quad |s_1't_1| \le  B_2, \quad  |r_1't_2| \le  B_3. 
$$
Multiplying $\Lambda_1$ by $r_1' t_2$, we obtain the following linear form 
\begin{equation}\label{eq:twologform}
r_1't_2 \Lambda_1 = 2h\log(\alpha_1^{r_1't_2}\cdot \alpha_2^{-s_1't_1}) - \log\left(\alpha_2^{\delta r_1's_1'} \cdot\alpha_3^{-r_1't_2}\right). 
\end{equation}
Thus, option \textbf{(A2)} yields a new linear form in two logarithms. To find a good lower bound for the new linear form \eqref{eq:twologform},
we apply a result due to Laurent \cite{Laurent:2008}. 

\begin{theorem}[Theorem 2 of \cite{Laurent:2008}]\label{thm:L1}
Let $a_1',a_2',h',\varrho$, and $\mu$ be real numbers with $\varrho>1$ and $1/3\le \mu \le 1$. Set
\begin{gather*}
\sigma = \frac{1+2\mu-\mu^2}{2},\quad \lambda' = \sigma \log \varrho,\quad H= \frac{h'}{\lambda'} + \frac{1}{\sigma},\\
\omega=2\left(1+\sqrt{1+\frac{1}{4H^2}}\right),\quad \theta = \sqrt{1+\frac{1}{4H^2}}+\frac{1}{2H}.
\end{gather*}
Consider the linear form 
$$
\Lambda = b_2 \log \gamma_2- b_1 \log\gamma_1,
$$
where $b_1$ and $b_2$ are positive integers. Suppose that $\gamma_1$ and $\gamma_2$ are multiplicatively independent.
Put $D=[\QQ(\gamma_1,\gamma_2):\QQ]/[\RR(\gamma_1,\gamma_2):\RR]$, and assume that 
\begin{align*}
h' &\ge\max\left\{D\left(\log\left(\frac{b_1}{a'_2}+\frac{b_2}{a'_1}\right)+\log\lambda' +1.75\right)+0.06,\lambda',\frac{D\log 2}{2}\right\}, \\ 
a_i'&\ge \max\{1,\varrho|\log \gamma_i|-\log|\gamma_i|+2Dh(\gamma_i)\}\quad (i=1,2),\\
a_1'a_2'& \ge {\lambda'}^2.
\end{align*}
Then 
$$
\log|\Lambda| \ge -C\left(h'+\frac{\lambda'}{\sigma}\right)^2a_1'a_2' - 
\sqrt{\omega\theta}\left(h'+\frac{\lambda'}{\sigma}\right)-\log\left(C'\left(h'+\frac{\lambda'}{\sigma}\right)^2a_1'a_2'\right)
$$
with 
$$
C=\frac{\mu}{\lambda'^3\sigma}\left(\frac{\omega}{6}+\frac{1}{2}\sqrt{\frac{\omega^2}{9}+
\frac{8\lambda'\omega^{5/4}\theta^{1/4}}{3\sqrt{a_1'a_2'}H^{1/2}}+\frac{4}{3}\left(\frac{1}{a_1'}+\frac{1}{a_2'}\right)\frac{\lambda'\omega}{H}}\right)^2,
\quad C'=\sqrt{\frac{C\sigma\omega\theta}{\lambda'^3\mu}}.
$$
\end{theorem}

In order to apply Theorem~\ref{thm:L1} to $-r_1't_2\Lambda_1$, we consider 
$$
D=4,\quad b_1 = 1,\quad b_2 =2h,\quad \gamma_1 = \alpha_2^{\delta r_1's_1'} \cdot\alpha_3^{-r_1't_2},
\quad \gamma_2 =  \alpha_1^{r_1't_2}\cdot \alpha_2^{-s_1't_1}. 
$$

Since $\alpha_1$, $\alpha_2$ and $\alpha_3$ are multiplicatively independent due to Lemma \ref{lem:mult_ind}, 
$\gamma_1$ and $\gamma_2$ are also multiplicatively independent and we may apply Theorem~\ref{thm:L1}.

Also note that a result coming from option \textbf{(A2)}, which surpasses \eqref{eq:hh} would not effect the final result.
Therefore, we may assume that $h$ is large, i.e. we may assume that $h\geq 3.8164\cdot 10^{9}\log \alpha_2 \log c$. We will keep
this assumption for the rest of the study of option \textbf{(A2)}.

Next we have to compute the heights and absolute values of the logarithms of $\gamma_1$ and $\gamma_2$. We obtain   
\begin{align*}
h(\gamma_1) &\le |\delta r_1's_1' | h(\alpha_2)  + |r_1't_2| h(\alpha_3)\\
&\le 0.5B_1\log \alpha_2 + B_3 \log c \le 212.2811 \log \alpha_2\log c,\\
h(\gamma_2) &\le |r_1't_2| h(\alpha_1) + |s_1't_1|h(\alpha_2)\\
&<0.5B_2 \log \alpha_1 + 0.5 B_3 \log \alpha_2\le 160.8641 (\log \alpha_2)^2.\\
|\log \gamma_1| &\le |\delta s_1' r_1'| \log \alpha_2  + |t_2r_1'| \log \alpha_3\\
&\le B_1 \log \alpha_2  + 0.46 B_3 \le 102.794 \log \alpha_2\log c. 
\end{align*}
In order to get a sharp upper bound for $|\log\gamma_2|$ we have a closer look on the linear form in logarithms \eqref{eq:twologform}. Since $|\Lambda_1|<1$, we obtain the inequality
$$|\log \gamma_2|<\frac{B_3+|\log(\gamma_1)|}{2h}<1.5\cdot 10^{-8},$$
by the previous estimates for $B_3$ and $|\log(\gamma_1)|$ together with our assumption on $h$.

In our next step we consider the quantities $a_1'$ and $a_2'$. We choose these quantities such that
\begin{align*}
a_1' &\geq 102.794 (\varrho+1) \log \alpha_2 \log c+ 1698.2488 \log \alpha_2\log c\\
& >(\varrho +1) |\log \gamma_1| + 8 h(\gamma_1)\\   
a_2' &\geq 1.5\cdot 10^{-8}(\varrho+1)+1286.9128 (\log \alpha_2)^2\\
&>(\varrho +1) |\log \gamma_2| + 8 h(\gamma_2). 
\end{align*}

If we choose $\varrho=52$ and $\mu = 0.61$, then we get $\sigma =  0.92395$ and $\lambda' =  3.65075\dots <3.651$. 
In view of these choices, we take
 $$
 a_1'= 7146.331 \log \alpha_2 \log c, \quad a_2'= 1286.913(\log\alpha_2)^2.
 $$
 
If we introduce the quantity
 $$F:=\frac{2.798639\cdot10^{-4} h}{\log \alpha_2 \log c}\ge  \frac{b_1}{a_2'} + \frac{b_2}{a_1'}$$
then we can write
$$h':= 4\log F+12.2398.$$
On the other hand, we assume that $h\geq 3.8164\cdot 10^{9}\log \alpha_2 \log c$, which yields $F>1.068\cdot 10^6$.
Thus, we may assume that 
$$H=\frac{h'}{\lambda'} + \frac{1}{\sigma}>19.6429.$$
The lower bound for $H$ gives us now the following upper bounds:
$$
\omega<4.00065,\quad \theta<1.02578, \quad C<0.02413,\quad C'<0.05551. 
$$

Now we have computed all quantities to apply Theorem~\ref{thm:L1} and we get
\begin{align*}
\log |r_1't_2 \Lambda_1 |&> -  221916.53\left(h'+\frac{\lambda'}{\sigma}\right)^2(\log\alpha_2)^3\log c\\
&\phantom{00000} -2.0258\left(h'+\frac{\lambda'}{\sigma}\right)
-\log\left( 510509.17 \left(h'+\frac{\lambda'}{\sigma}\right)^2(\log\alpha_2)^3\log c\right)\\
&>-221916.6\left(h'+\frac{\lambda'}{\sigma}\right)^2(\log\alpha_2)^3\log c.
\end{align*}
By Lemma \ref{lem:low-Lamb1}, we have 
$$
\log |r_1't_2 \Lambda_1 |<\log B_3-4j\log\alpha_2 < \log B_3-4h \log \alpha_1
$$
 and since $\log\alpha_2 < 3\log \alpha_1$, we get 
 \begin{align*}
 h &<166437.45\left(h'+\frac{\lambda'}{\sigma}\right)^2(\log\alpha_2)^2\log c + \frac{\log(160.915\log \alpha_2)}{\log\alpha_1}\\
 &<166437.46\left(h'+\frac{\lambda'}{\sigma}\right)^2(\log\alpha_2)^2\log c.
\end{align*}

Multiplying the above inequality by $\frac{2.798639\cdot10^{-4}}{\log \alpha_2 \log c}$ and noting that $h'= 4\log F+12.2398$
we obtain
\begin{equation}\label{eq:F}
F<745.278\left(\log F+4.048\right)^2(\log\alpha_2).
\end{equation}
By Proposition~\ref{pro:1}, we have $ac<6.18\cdot 10^{32}$ and so $\log\alpha_2 < 38.446$. From \eqref{eq:F} we deduce that 
$$
F < 28652.96(\log F+4.048)^2, 
$$
which yields $F<1.18493\cdot 10^7$. Thus we get the inequality 
\begin{equation}\label{eq:upp1}
h<4.234\cdot 10^{10}\log \alpha_2 \log c.
\end{equation}
We use the inequality $h>6.2462\sqrt{ac}$ (cf. Lemma \ref{lem:hbd}) and
obtain from inequality \eqref{eq:upp1} that $ac<1.6\cdot 10^{26}$. This implies that $\log\alpha_2<30.8618$. 
Inserting this value again into \eqref{eq:F} we get
$$
F < 23000.63(\log F+4.048)^2
$$
and we get $F<9.2851\cdot 10^6$. Thus, we obtain a slightly improved version of inequality~\eqref{eq:upp1}
\begin{equation}\label{eq:upp2}
h<3.3178\cdot 10^{10}\log \alpha_2 \log c,
\end{equation}
which is unfortunately still weaker than inequality \eqref{eq:hh}. However, we
obtain from this last inequality that $ac<9.45\cdot 10^{25}$ and $h<6.08\cdot 10^{13}$.

Now, let us briefly discuss option \textbf{(A1)}. In this case, we similarly proceed as in the case \textbf{(A2)}.
However, in this case we apply Theorem~\ref{thm:L1} to 
$$
r_0\Lambda_1 = 2h\log\left(\alpha_1^{r_0} \alpha_2^{-s_0}\right) + r_0\log\alpha_3,
$$
with $|r_0|\le B_3$ and $|s_0|\le B_2$. Therefore, we choose $\gamma_1= \alpha_1^{r_0}$ and $\gamma_2=\alpha_3^{r_0}$.
Computing the values for $h(\gamma_i)$ and $|\log \gamma_i|$ for $i=1,2$ one obtains smaller values than
in the case \textbf{(A2)}. Therefore, in the case \textbf{(A1)} one obtains smaller upper bounds than indicated by~\eqref{eq:upp2}.

We apply Theorem \ref{thm:kit} in combination with Theorem \ref{thm:L1} two times more. Choosing $\rho=9$, $\chi=2$, $L=519$,
$M=14.02$ in Theorem \ref{thm:kit} and $\varrho=57$ and $\mu=0.61$ in Theorem \ref{thm:L1} together with the upper bounds
$ac<9.47\cdot 10^{25}$ and $h<6.08\cdot 10^{13}$ yields the slightly better bounds $ac<6.87\cdot 10^{25}$ and $h=5.18\cdot 10^{13}$. 
Now choosing $\rho=9$, $\chi=2$, $L=518$, $M=13.92$ in Theorem \ref{thm:kit} and $\varrho=56$ and $\mu=0.61$ in Theorem \ref{thm:L1} yields.

\begin{proposition}\label{pro:2}
  If $\{a,b,c,d,e\}$ is a Diophantine quintuple with $a < b < c < d < e$, then we have 
  $ac<6.77\cdot 10^{25}$, $d<1.83\cdot 10^{52}$. Provided that $c>2\cdot 10^8$, we also have
  \begin{equation}\label{eq:hh2}
h<2.8376 \cdot 10^{10} \log \alpha_2 \log c<5.136\cdot 10^{13}.
\end{equation}  
\end{proposition}

From a computational point of view these upper bounds are still too large to apply the Baker-Davenport reduction method directly.

\section{Euler triples} \label{sec:7}

In this section we still assume that $\{a,b,c,d,e\}$ is a Diophantine quintuple with $a<b<c<d<e$. However, we additionally assume that $\{a,b,c,d\}$ is an Euler quadruples, i.e. a Diophantine quadruple of the form $\{a,b, a+b+2r, 4r(a+r)(b+r)\}$. Particularly, we have 
$$s=a+r,\quad  t=b+r,$$ 
and 
$$x=at+rs,\quad y=rt+bs,\quad z=cr+st.$$
If we insert the relations for $s$ and $t$ into the relations for $x$, $y$ and $z$, then we obtain
\begin{align*}
 x=at+rs&=ab+2ar+r^2=2r^2+2ar-1\\
 &=2rs-1,\\
 y=rt+bs&=ab+2rb+r^2=2r^2+2rb-1\\
 &=2rt-1,\\
 z=cr+st&=ar+br+2r^2+st=ab+ar+br+r^2+st+1\\
 &=(a+r)(b+r)+st+1=2st+1.
\end{align*}
For the rest of this section, we will use these relations without any special reference. In the next lemma, we will use
these relations to obtain several congruence relations that will be crucial in this section.

\begin{lemma}\label{lem:lmn2}
Assume that $\{a,b,c,d,e\}$ is a Diophantine quintuple such that $\{a,b,c\}$ is an Euler triple. Then we have 
\begin{align*}
l &\equiv \frac{1-(-1)^j}{2} (-\varepsilon c) \pmod{s},& n &\equiv \frac{1-(-1)^j}{2} (\varepsilon a) \pmod{s},\\
m &\equiv \frac{1-(-1)^k}{2} (-\varepsilon c) \pmod{t},& n &\equiv \frac{1-(-1)^k}{2} (\varepsilon b) \pmod{t}.
\end{align*}
\end{lemma}

\begin{proof}
If we consider relation \eqref{eq:acjj} modulo $2s$, then we obtain 
\begin{equation}\label{eq:acs}
\begin{split}
Z\sqrt{a} + X\sqrt{c}&=(\sqrt{a}+\sqrt{c})(s+\sqrt{ac})^{2j}\\ 
&=(\sqrt{a}+\sqrt{c})(2s^2-1+2s\sqrt{ac})^{j}\\ 
&\equiv (\sqrt{a}+\sqrt{c})(-1)^j\\ 
&= (-1)^j \sqrt{a} + (-1)^j\sqrt{c} \pmod{2s}.
\end{split}
\end{equation}
We immediately get
\begin{equation}\label{eq:alpha1}
X = X^{(a,c)}_{2j} \equiv (-1)^j \pmod{2s}
\end{equation}
and 
\begin{equation}\label{eq:gamma1}
Z= Z^{(a,c)}_{2j} \equiv (-1)^j \pmod{2s}. 
\end{equation}

To obtain further congruences for $X$ and $Z$, we consider \eqref{eq:adll} and \eqref{eq:cdnn} modulo $2d$ and obtain 
\begin{equation}\label{eq:add}
\begin{split}
W \sqrt{a} + X \sqrt{d} &= (\varepsilon \sqrt{a} + \sqrt{d})(x+\sqrt{ad})^{2l}\\ 
&=(\varepsilon\sqrt{a}+\sqrt{d})(2ad+1+2x\sqrt{ad})^{l}\\ 
&\equiv (\varepsilon\sqrt{a}+\sqrt{d})(1+2x\sqrt{ad})^l\\ 
&\equiv (\varepsilon\sqrt{a}+\sqrt{d})(1+2lx\sqrt{ad}) \\ 
&\equiv \varepsilon\sqrt{a}+(1+2\varepsilon a  x l)\sqrt{d} \pmod{2d}
\end{split}
\end{equation}
and
\begin{equation}\label{eq:cdd}
\begin{split}
W\sqrt{c} +Z\sqrt{d}&=(\varepsilon\sqrt{c}+\sqrt{d})(z+\sqrt{cd})^{2n}\\ 
&= (\varepsilon\sqrt{c}+\sqrt{d})(2cd+1+2z\sqrt{cd})^{n}\\ 
&\equiv (\varepsilon\sqrt{c}+\sqrt{d})(1+2z\sqrt{cd})^n\\ 
&\equiv (\varepsilon\sqrt{c}+\sqrt{d})(1+2nz\sqrt{cd}) \\ 
&\equiv \varepsilon\sqrt{c}+(1+2\varepsilon cz n)\sqrt{d} \pmod{2d},
\end{split}
\end{equation}
respectively. Therefore, we get some further congruences for $X$ and $Z$
\begin{equation}\label{eq:alpha2}
 X= X^{(a,d)}_{2l} \equiv 1+2\varepsilon a  x l \pmod{2d}
\end{equation}
and
\begin{equation}\label{eq:gamma2}
 Z= Z^{(c,d)}_{2n} \equiv 1+2\varepsilon czn \pmod{2d}.
\end{equation}

Since $d=4rst$ we take \eqref{eq:alpha2} modulo $2s$ and 
combining this congruence with \eqref{eq:alpha1} we get 
$$
1+2\varepsilon a  x l \equiv (-1)^j \pmod{2s},
$$
which implies 
$$
\varepsilon axl \equiv \frac{(-1)^j-1}{2} \pmod{s}.
$$
As $ac\equiv-1 \pmod{s}$, we have
$$
xl \equiv \frac{1-(-1)^j}{2}(\varepsilon c) \pmod{s}.
$$
Since $x=2rs-1\equiv -1 \pmod s$, we finally get 
$$
l \equiv \frac{1-(-1)^j}{2}(-\varepsilon c) \pmod{s},
$$
which is the first congruence of Lemma \ref{lem:lmn2}.

Similarly, we consider the congruences \eqref{eq:gamma1} and \eqref{eq:gamma2} and obtain 
$$
1+2\varepsilon czn \equiv (-1)^j \pmod{2s}.
$$
Thus, we deduce 
$$
\varepsilon czn \equiv \frac{(-1)^j-1}{2} \pmod{s}.
$$
Again, as $ac\equiv-1 \pmod{s}$, we get
$$
zn \equiv \frac{1-(-1)^j}{2}(\varepsilon a) \pmod{s}
$$
and since $z=2st+1\equiv 1 \pmod s$, we have 
$$
n \equiv \frac{1-(-1)^j}{2}(\varepsilon a) \pmod{s},
$$
which yields the second congruence of Lemma \ref{lem:lmn2}. Therefore, we have achieved the first half of Lemma \ref{lem:lmn2}. 

For the proof of the second half of Lemma \ref{lem:lmn2} we consider \eqref{eq:bckk} modulo $2t$ and get 
$$
Z\sqrt{b} + Y\sqrt{c}  \equiv (-1)^k\sqrt{b} + (-1)^k\sqrt{c} \pmod{2t}.
$$
From the above congruence, we obtain two congruences for $Y$ and $Z$:
\begin{equation}\label{eq:beta1}
Y = Y^{(b,c)}_{2k} \equiv (-1)^{k} \pmod{2t}
\end{equation}
and 
\begin{equation}\label{eq:gamma3}
Z= Z^{(b,c)}_{2k}\equiv (-1)^{k}\pmod {2t}. 
\end{equation}
Similarly to \eqref{eq:alpha2}, we get
\begin{equation}\label{eq:beta2}
Y = Y^{(b,d)}_{2m} \equiv  1+2\varepsilon b  y m \pmod{2d}.
\end{equation}

Combining \eqref{eq:beta1} and \eqref{eq:beta2}, we deduce by similar arguments as above that
$$
1+2\varepsilon bym \equiv (-1)^k \pmod{2t}.
$$
Since $bc\equiv -1 \pmod t$ and $y\equiv -1\pmod{t}$, we obtain
$$
m\equiv \frac{1-(-1)^k}{2}(-\varepsilon c) \pmod{t},
$$
which establishes the third congruence of Lemma \ref{lem:lmn2}.

Finally, we consider the congruences \eqref{eq:gamma2} and \eqref{eq:gamma3} and obtain
$$
1+2\varepsilon czn \equiv (-1)^k \pmod{2t}.
$$
Using the congruences $bc\equiv -1 \pmod{t}$ and $z\equiv 1\pmod{t}$, we have 
$$
n \equiv \frac{1-(-1)^k}{2}(\varepsilon b) \pmod{t}.
$$
This completes the proof of Lemma \ref{lem:lmn2}.
\end{proof}

\begin{lemma}\label{lem:lmnst}
 Suppose that $\{a,b,c,d,e\}$ is a Diophantine quintuple such that $\{a,b,c\}$ is an Euler triple. 
 Then, at least one of the following congruences holds 
\begin{description}
\item[I\phantom{II}] $l\equiv n\equiv 0\pmod{s}$,
\item[II\phantom{I}]  $m\equiv n\equiv 0 \pmod{t}$,
\item[III]  $n\equiv -\varepsilon r \pmod{st}$.
\end{description}
\end{lemma}

\begin{proof}
If $j$ is even, then $\frac{1-(-1)^j}{2}=0$ and by Lemma~\ref{lem:lmn2}, we have
$$
l\equiv n\equiv 0\pmod{s},
$$
which yields case I.

If $k$ is even, then we similarly get 
$$
m\equiv n\equiv 0 \pmod{t},
$$
which yields case II.

Therefore, we may assume that both $j$ and $k$ are odd. Thus, Lemma~\ref{lem:lmn2} provides us with 
$$
 n\equiv \varepsilon a \pmod{s}, \quad  n \equiv  \varepsilon b   \pmod{t}.
$$
Since $s=a+r$ and $t=b+r$, we have 
$$
 n\equiv -\varepsilon r \pmod{s},
\quad  n \equiv  -\varepsilon r   \pmod{t}.
$$
As
$$
\gcd(s,t)=\gcd(s,s+t)=\gcd(s,c)=1,
$$ 
we have by the Chinese remainder theorem 
$$
n\equiv -\varepsilon r \pmod{st}.
$$
Therefore, Lemma \ref{lem:lmnst} is proved completely.
\end{proof}

Our next aim is to show that the options I and II of Lemma \ref{lem:lmnst} yield only small solutions. Such a result will be
achieved by using linear forms in logarithms. In particular, let
\begin{gather*}
\beta_1= x+\sqrt{ad},\quad \beta_2 = y+ \sqrt{bd},\quad \beta_3 = z+\sqrt{cd},\\
\beta_4 = \frac{\sqrt{c}(\varepsilon\sqrt{a}+\sqrt{d})}{\sqrt{a}(\varepsilon\sqrt{c}+\sqrt{d})}, \quad \beta_5=\frac{\sqrt{c}(\varepsilon\sqrt{b}+\sqrt{d})}{\sqrt{b}(\varepsilon\sqrt{c}+\sqrt{d})}. 
\end{gather*}
Then, we consider the two linear forms in logarithms
$$
\Lambda_2=2l\log\beta_1 - 2n\log\beta_3+\log\beta_4,
$$
and
$$
\Lambda_3= 2m\log\beta_2- 2n\log\beta_3+\log\beta_5.
$$
As a first step to obtain upper bounds for $s$ and $t$ in case I and II we will consider the linear forms $\Lambda_2$ and $\Lambda_3$ in case I and II respectively.

First, let us establish good upper bounds for these linear forms in logarithms. This can be done
with the help of \cite[Lemma 5]{Dujella:2001}. Instead of considering the Diophantine quadruple $\{a,b,c,d\}$,
we consider the Diophantine quadruples $\{a,b,d,e\}$ and $\{b,c,d,e\}$ respectively. Finally, note that the condition
$c>4b$ in \cite[Lemma 5]{Dujella:2001} is in our situation fulfilled since $d>4abc>4b$ and $d>4abc>4c$ respectively. Therefore, we
obtain the following lemma.

\begin{lemma}\label{lem:uppbd}
We have
$0<\Lambda_2<\frac{8}{3}ad\cdot \beta_1^{-4l}$ and $0<\Lambda_3<\frac{8}{3}bd\cdot \beta_2^{-4m}$. 
\end{lemma}

The next lemma will deal with the case I. 

\begin{lemma}\label{lem:euler1}
If $l\equiv n\equiv 0 \pmod{s}$, then $s<20493$. 
\end{lemma}

\begin{proof}
In the case that $r\leq 10000$ we get 
$$s=r+a< 2r\leq 20000.$$
Therefore, we may assume that $r>10000$. Let $l=sl_1, n=sn_1$, for some positive integers $l_1,n_1$. We rewrite $\Lambda_2$ into the form
$$
\Lambda_2= \log\beta_4-2s\log \left({\beta_3^{n_1}}/{\beta_1^{l_1}}\right).
$$
In view of an application of Laurent's result Theorem \ref{thm:L1} we set 
$$D=4,\quad b_1=2s, \quad b_2=1, \quad  \gamma_1= {\beta_3^{n_1}}/{\beta_1^{l_1}}, \quad \gamma_2 = \beta_4.$$

 Next we want to estimate the height of $\gamma_1$. Since $\beta_1$ and $\beta_3$ are units in $\QQ(\sqrt{ad})$ and $\QQ(\sqrt{bd})$ respectively, we deduce that
$\gamma_1$ is an algebraic integer, i.e. the leading coefficient of its minimal polynomial is $1$, and 
the conjugates of $\gamma_1$ are
$$\frac{\beta_3^{n_1}}{\beta_1^{l_1}},\quad \frac{\beta_3^{-n_1}}{\beta_1^{l_1}},
\quad \frac{\beta_3^{n_1}}{\beta_1^{-l_1}},\quad \frac{\beta_3^{-n_1}}{\beta_1^{-l_1}}. $$
Depending on whether ${\beta_3^{n_1}}>{\beta_1^{l_1}}$ or ${\beta_3^{n_1}}<{\beta_1^{l_1}}$ we have
$$h(\gamma_1) = \frac 14\left(\left|\log \frac{\beta_3^{n_1}}{\beta_1^{l_1}}\right|+\left|\log\frac{\beta_3^{n_1}}{\beta_1^{-l_1}}\right|\right)=
\frac{n_1}{2}\log \beta_3$$
or
$$h(\gamma_1) = \frac 14\left(\left|\log \frac{\beta_3^{-n_1}}{\beta_1^{-l_1}}\right|+\left|\log\frac{\beta_3^{n_1}}{\beta_1^{-l_1}}\right|\right)=
\frac{l_1}{2}\log \beta_1$$
respectively. On the other hand, the definition of $\Lambda_2$ together with Lemma~\ref{lem:uppbd} yields
$$
\left|\log \left({\beta_3^{n_1}}/{\beta_1^{l_1}}\right)\right| < \frac{1}{s}\left(\log\beta_4 + \frac{8}{3}ad\beta_1^{-4}\right)
<\frac{1}{s}\left(\log\beta_4+\frac{1}{ad}\right).$$
Claiming that $\beta_4<2\sqrt{\frac ca}$, we obtain
$$
\left|\log \left({\beta_3^{n_1}}/{\beta_1^{l_1}}\right)\right| < \frac{\log \left(2\sqrt{c/a}\right)}s+\frac{1}{sad}<\frac{\log 2s}s+\frac{1}{sad}<0.001
$$
since we assume that $s>r>10000$. Thus, in any case, we have  
$$
h(\gamma_1) <\frac{l_1}{2}\log \beta_1 +0.001. 
$$

To justify our claim that $\beta_4<2\sqrt{\frac ca}$, we compute
\begin{align*}
 \beta_4 &= \sqrt{\frac ca}\frac{\varepsilon\sqrt{a}+\sqrt{d}}{\varepsilon\sqrt{c}+\sqrt{d}}
 = \sqrt{\frac ca} \left(1-\varepsilon\frac{\sqrt{c}-\sqrt{a}}{\sqrt{d}+\varepsilon\sqrt{c}}\right)\\
 &\leq \sqrt{\frac ca} \left(1+\frac{\sqrt{c}}{\sqrt{d}-\sqrt{c}}\right)
 \leq \sqrt{\frac ca} \left(1+\frac{2\sqrt{c}}{\sqrt{d}}\right) \qquad \text{since}\;\; d>4abc>4c\\
 &\leq 2\sqrt{\frac ca}. 
\end{align*}

Next, we compute the height of $\gamma_2=\beta_4$. All the absolute values of conjugates of $\beta_4$, namely 
$$
\frac{\sqrt{c}(\sqrt{d}+\sqrt{a})}{\sqrt{a}(\sqrt{d}+\sqrt{c})},\quad \frac{\sqrt{c}(\sqrt{d}+\sqrt{a})}{\sqrt{a}(\sqrt{d}-\sqrt{c})},\quad
\frac{\sqrt{c}(\sqrt{d}-\sqrt{a})}{\sqrt{a}(\sqrt{d}+\sqrt{c})},\quad \frac{\sqrt{c}(\sqrt{d}-\sqrt{a})}{\sqrt{a}(\sqrt{d}-\sqrt{c})}, 
$$
are greater than one. Moreover, the minimal polynomial of $\beta_4$ is 
$$
a^2(d-c)^2X^4+4a^2c(d-c)X^3-2ac(d^2+db+dc-3ac)X^2+4ac^2(d-a)X+c^2(d-a)^2.
$$
Note that the minimal polynomial does not depend on $\epsilon$ since $\frac{\sqrt{c}(\sqrt{d}+\sqrt{a})}{\sqrt{a}(\sqrt{d}+\sqrt{c})}$
and $\frac{\sqrt{c}(\sqrt{d}-\sqrt{a})}{\sqrt{a}(\sqrt{d}-\sqrt{c})}$ are algebraic conjugates. Thus we obtain
$$
h(\gamma_2)=h(\beta_4) =\frac{1}{4}\log\left(a^2(d-c)^2\cdot \frac{c^2}{a^2} \cdot \frac{(d-a)^2}{(d-c)^2}\right)< \frac{1}{2}\log(cd) <\log \beta_3. 
$$

Choosing $\varrho=61$ and $\mu = 0.7$ in Theorem~\ref{thm:L1} we get $\sigma = 0.955$, $\lambda' = 3.92588...<3.93$. Moreover, we take
\begin{align*}
 a_1'&:=4l_1 \log\beta_1 +0.07\\
 &>0.001(\varrho+1)  +0.008+ 4l_1 \log\beta_1 \\
& \geq  \varrho|\log\gamma_1| - \log|\gamma_1| + 8h(\gamma_1).
 \end{align*}
 Since $c=a+b+2r<4b$, we have $d>4abc>c^2$ and $\beta_3>2\sqrt{cd}>2c^{3/2}$. Therefore we obtain
 \begin{align*}
a_2'&:=28\log \beta_3> (\varrho-1)\log(2^{1/3}\sqrt{c}) + 8\log \beta_3\\
&>(\varrho-1)\log\left(\frac{\sqrt{c}(\sqrt{d}+\epsilon\sqrt{a})}{\sqrt{a}(\sqrt{d}+\epsilon\sqrt{c})}\right)+8\log \beta_3\\
&>\varrho|\log\gamma_2| - \log|\gamma_2| + 8h(\gamma_2).
 \end{align*} 
Note that the second inequality holds since
$$
\frac{(\sqrt{d}+\sqrt{a})}{\sqrt{a}(\sqrt{d}+\sqrt{c})}\leq \frac{(\sqrt{d}-\sqrt{a})}{\sqrt{a}(\sqrt{d}-\sqrt{c})} \leq 1+\frac{\sqrt{c}-\sqrt{a}}{\sqrt{d}-\sqrt{c}}
 1+\frac{\sqrt{c}}{\sqrt{60c}-\sqrt{c}}<\frac 76<2^{1/3}.
$$
In particular note that $d>4abc\geq 60c$ since $b\geq 15$ due to Lemma \ref{lem:min_bcd}.

Since $r>10000$ we have
$$a_1'>4\log(2\sqrt{ad})>4\log(2\sqrt{4rst})>4\log(4(r^3)^{1/2})>60$$ 
and 
 $$a_2'>28\log(2\sqrt{cd})>28\log(2(16r^4)^{1/2})>574. $$
It is easy to see that $a'_1a'_2>\lambda'^2$ and our choice of parameters is admissible.

We set 
$$b':=\frac{2s}{a_2'}+0.02>\frac{2s}{a_2'} + \frac{1}{4l_1\log \beta_3+0.07}=\frac{b_1}{a_2'} + \frac{b_2}{a_1'}$$
and choose
\begin{align*}
h'&=4\log b'+  12.6\\
&\geq \max\left\{D\left(\log\left(\frac{b_1}{a'_2}+\frac{b_2}{a'_1}\right)+\log\lambda' +1.75\right)+0.06,\lambda',\frac{D\log 2}{2}\right\}.
\end{align*}
Therefore, as $\beta_3 = z+\sqrt{cd}<2z$ and  $z=2st+1<2s^3+1$, we have
\begin{align*}
h'&=4\log\left(\frac{s}{14\log\beta_3}+0.02\right)+12.6\\
&>4\log\left(\frac{s}{14\log(4s^3+2)}\right)+12.6>25.4.
\end{align*}
This implies that $H>7.5$. Then, we have $\omega<4.01$, $\theta<1.07$. And hence 
$$B:= \frac{1}{4}\left(h'+\frac{\lambda'}{\sigma}\right)<\log b'+4.2. $$
We obtain   
$$
C<0.0226,\quad C'<0.047.
$$

Now, we apply Theorem~\ref{thm:L1} and get 
\begin{equation}\label{eq:lola1}
\log|\Lambda_1| \ge -0.3616B^2a_1'a_2' -8.29B - \log (0.76B^2 a_1'a_2').
\end{equation}
By Lemma~\ref{lem:uppbd}, we have
\begin{equation}\label{eq:ula1}
\log |\Lambda_1| <-4sl_1\log \beta_1 + \log \left(\frac{8}{3}ad\right)=-s(a'_1-0.07)+ \log \left(\frac{8}{3}ad\right).
\end{equation}
Combining \eqref{eq:lola1} and \eqref{eq:ula1}, we obtain
\begin{equation}
s(a'_1-0.07)<0.3616 B^2a_1'a_2' +8.29B + \log (0.76B^2 a_1'a_2')+\log \left(\frac{8}{3}ad\right).
\end{equation}
Dividing both sides of the above inequality by $a_1'a_2'$ and simplifying, we obtain
$$
b'<0.725(\log b' + 4.2)^2. 
$$
Thus, we get $b'<46.98$.
As $b'=\frac{2s}{a_2'}+0.02$, we deduce that 
$$
s<23.48a_2'=23.48(28\log \beta_3+0.04)<657.44\log(4s^3+2)+1. 
$$
Therefore, we obtain $s<20493$.
\end{proof}

\begin{lemma}\label{lem:euler2}
If $m\equiv n\equiv 0 \pmod{t}$, then $t<22023$. 
\end{lemma}

\begin{proof}
 The proof is similar to that of Lemma~\ref{lem:euler1}. However we may only assume that $r \geq 145$. Indeed assuming $r<145$ yields $t=r+b<r+r^2<21170$.
 Note that assuming $r\leq 150$ would result in $t\leq 22650$ which already succeeds the bound given in the lemma.

By the assumption of the lemma we have that $m\equiv n\equiv 0 \pmod{t}$, and we may 
write $m=tm_2, n=tn_2$ for some positive integers $m_2,n_2$. We rewrite $\Lambda_3$ into the form
$$
\Lambda_3 = \log\beta_5 - 2t\log \left({\beta_3^{n_2}}/{\beta_2^{m_2}}\right)
$$
and apply Theorem \ref{thm:L1} to this linear form. As the application of Theorem \ref{thm:L1} is technically similar to that in the proof of  Lemma~\ref{lem:euler1}, we omit the details.
We only want to note that the slightly larger upper bound is due to the fact that we only assume that $r\geq 145$. Hence we 
obtain smaller lower bounds for $a'_1,a'_2,h',H$ and so on. Therefore, we obtain slightly larger upper bounds for $C,C'$ and so on resulting in a slightly larger upper bound for $t$. 
\end{proof}

Next, we consider case III

\begin{lemma}\label{lem:euler3}
If $n\equiv -\varepsilon r \pmod{st}$, then we have $r<900154$ and $h<9.6 \cdot 10^{15}$.
\end{lemma}

\begin{proof} By Lemma~\ref{lem:lmn} and Lemma~\ref{lem:mdb}, we have 
$$
n\ge \frac m2\geq \frac{1}{2}\cdot \frac{\sqrt{17}-1}{2}\cdot \sqrt{\frac{d}{b}}\ge \frac{\sqrt{17}-1}{2}\sqrt{ac}>s>r.
$$
If $n\equiv -\varepsilon r \pmod{st}$, then $n+\varepsilon r \ge st$, hence 
$n\ge st -r\ge c(r-1)$. On the other hand, we have by Lemma~\ref{lem:lmn} and Lemma~\ref{lem:hm} that $h\ge 2n$ and therefore, we get 
$$h\ge 2c(r-1).$$

In view of the statement of the lemma we may assume that $r>900000$.  This implies that $c=a+b+2r>10^6$.  Due to Proposition~\ref{pro:2}, we have
$$
h<2.8376\cdot 10^{10} \log (s+\sqrt{ac}) \log c.
$$
Combining the upper and lower bound for $h$, we have 
$$
c(r-1)\le 1.4188 \cdot 10^{10} \log(2\sqrt{c(r-1)})\log(c(r-1)/900000).
$$
This implies that
$$c(r-1)<3.233\cdot 10^{12}.$$
Since $c=a+b+2r\ge 2\sqrt{ab}+ 2r >3.99r$ we have
$$3.99r(r-1)< 3.233\cdot 10^{12}$$
and therefore, we obtain $r<900154$.

We are left to compute the upper bound for $h$. If $c>2\cdot 10^8$, then by Proposition \ref{pro:2} we have
\begin{multline*}
h<2.8376\cdot 10^{10} \log (s+\sqrt{ac}) \log c<2.8376\cdot 10^{10}\log(2(a+r))\log(a+b+2r)\\<2.8376\cdot 10^{10}\log(4r-2)\log(1+r^2+2r)<1.2\cdot 10^{13}.
\end{multline*}
If $c\le 2\cdot 10^8$, then we get the following inequalities coming from inequality \eqref{eq:nd}
$$
\frac{h}{\log(38.92h)}< 1.232\cdot 10^{12}\cdot \log (2\sqrt{ac+1})\cdot \log c<4.51 \cdot 10^{14}
$$ 
and so $h<1.9\cdot 10^{16}$. 
\end{proof}

Now, assume that $\{a,b,c,d,e\}$ is a Diophantine quintuples such that $a<b<c<d<e$ and $\{a,b,c\}$ is an Euler triple. Then, by Lemma~\ref{lem:lmnst} and Lemmas~\ref{lem:euler1}--\ref{lem:euler3}, we have 
$$
r<900154,\quad \text{and} \quad h<1.9\cdot 10^{16}. 
$$
Note that the upper bounds obtained from Lemma~\ref{lem:euler1} and Lemma~\ref{lem:euler2} are much smaller than the bound obtained from Lemma~\ref{lem:euler3}.

In order to deal with the remaining cases, we
will use a Diophantine approximation algorithm called the
Baker-Davenport reduction method. The following lemma is a slight
modification of the original version of the Baker-Davenport reduction
method (see \cite[Lemma 5a]{Dujella-Pethoe:1998}).

\begin{lemma}\label{lem:Baker-Davenport}
Assume that $M$ is a positive integer. Let $p/q$ be the convergent
of the continued fraction expansion of a real number $\kappa$ such that $q > 6M$ and let
$$\eta=\left\| \mu q \right\| - M \cdot \| \kappa q\|,$$
where $\parallel\cdot\parallel$ denotes the distance from the
nearest integer. If $\eta > 0$, then the
inequality
$$0 < J\kappa -K+\mu < AB^{-J}$$
has no solutions in integers $J$ and $K$ with
$$
\frac{\log\left(Aq/\eta\right)}{\log B}\leq J \leq M.
$$
\end{lemma}

We apply Lemma~\ref{lem:Baker-Davenport} to 
$$
\Lambda_1 = 2h \log(r+\sqrt{ab})- 2j \log(s+\sqrt{ac}) + \log \left(\frac{\sqrt{c}(\sqrt{a}+\sqrt{b})}{\sqrt{b}(\sqrt{a}+\sqrt{c})}\right).
$$
with $s=a+r$, $c=a+b+2r$ and
\begin{gather*}
\kappa=\frac{\log(r+\sqrt{ab})}{\log(s+\sqrt{ac})},\quad
\mu=\frac{\log \left(\frac{\sqrt{c}(\sqrt{a}+\sqrt{b})}{\sqrt{b}(\sqrt{a}+\sqrt{c})}\right)}{\log (s+\sqrt{ac})},\\
 A=\frac{1}{\log(s+\sqrt{ac})},\quad B=(r+\sqrt{ab})^2
\end{gather*}
and $J= 2h$, $M =1.9\cdot 10^{16}$. 
We ran a GP program to check all $58258307$ pairs $(a,b)$ such that $2\le r\le 900153$ and obtained $J\le15$ in each case.
This contradicts the fact that $J=2h\ge 4c(r-1)\ge 48$. It took $7$ hours and $2$ minutes to run the program on a MacBook Pro 
with an i7 4960hq CPU and 16G memory. Summarizing our results we obtain:

\begin{theorem}\label{thm:euler}
An Euler triple $\{a,b,a+b+2\sqrt{ab+1}\}$ cannot be extended to a Diophantine quintuple.
\end{theorem}

\section{Non-Euler triples} \label{sec:8}
In this section, we will deal with the non-Euler triples. We will consider two cases: 
the case that the degree is one and the case that the degree is greater than one.
We start with the case that the degree is one.

\begin{theorem}\label{thm:eulerplus}
A Diophantine triple $\{a,b,c\}$ cannot be extended to a Diophantine quintuple if $\deg(a,b,c)=1$.  
\end{theorem}

\begin{proof}
Assume that $a<b<c$. If $\deg(a,b,c)=1$, then the triple $\{d_{-1},a, b\}$ is an Euler triple, where
$$
d_{-1} = d_{-}(a,b,c)
$$
 Moreover the quadruple $\{d_{-1},a,b,c\}$ is regular due to Proposition \ref{pro:2-2}. Therefore we have 
$$
d_{-1} = a+b \pm 2r
$$
and 
$$
c= d_{+}(a,d_{-1},b) = 4r(r\pm a)(b\pm r). 
$$

We assume that $b>10000$. Then, $d_{-1} \ge a+b - 2\sqrt{\frac{b^2}{24}+1} > 0.59b$ and together with the observation that $c>4abd_{-1}$, we have 
$$
2.36 (ab)^2<4a^2bd_{-1}<ac < 6.77 \cdot 10^{25}. 
$$ 
Hence, we get $ab<5.36\cdot 10^{12}$ and $r\le 2315167$. Moreover, we have 
$a<\left(\frac{r^2}{2}\right)^{\frac{2}{5}} <93596$ since $b>\max\{24a,2a^{3/2}\}$ due to Lemma \ref{lem:b3a}. 

We apply Lemma~\ref{lem:Baker-Davenport} to $\Lambda_1$ and check $109748916$ pairs $(a,r)$ such that
 $$
b=(r^2-1)/{a},\quad c= 4r(r\pm a)(b\pm r), 
$$
and $\kappa, \mu, A, B$ are taken as in the previous section. Moreover, we choose $J=2h$ and $M=1.9\cdot 10^{16}$.
The running time of the GP program is less than 16 hours. In all $219497932$ cases, we have $J\le 15$. This contradicts the fact that $J=2h>10\sqrt{ac}>20\sqrt{2}$.

\end{proof}

Now we have to deal with the case that the degree of the triple is greater than one. In this case we prove the following theorem.

\begin{theorem}\label{thm:deg2}
A Diophantine triple $\{a,b,c\}$ cannot be extended to a Diophantine quintuple if $\deg(a,b,c)\ge2$.  
\end{theorem}

\begin{proof}
Let us consider a Diophantine triple $\{a,b,c\}$ with $\deg(a,b,c)\ge 2$. Using Proposition \ref{pro:2}, we may assume that 
$ac<6.77 \cdot 10^{25}$. Since $\deg(a,b,c)\ge 2$, there exist positive integers $d_{-1}$ and $d_{-2}$ satisfying 
$$
d_{-1} = d_{-}(a,b,c),\quad d_{-2} = d_{-}(a,b,d_{-1}). 
$$ 
By Lemma~\ref{lem:Jones}, since $\{a,b,c\}$ is not an Euler triple we have $c>a+b+2\sqrt{ab+1}$ and therefore $c>4ab$.
Furthermore, we have $ac<0.927b^3$ due to Lemma~\ref{lem:acb}. These observations imply
$$
4ab<c<180.45\frac{b^3}{a}. 
$$ 
For the remainder of the proof, we split up the interval $\left(4ab,180.45\frac{b^3}{a}\right)$ into five subintervals:
$$
c\in \left(4ab, 4a^{\frac{1}{2}}b^{\frac{3}{2}}\right]\cup \left(4a^{\frac{1}{2}}b^{\frac{3}{2}},4ab^2\right]\cup \left(4ab^2,4ab^{\frac{5}{2}}\right]\cup \left(4a^{\frac{3}{2}}b^{\frac{5}{2}}, {4a^2b^3}\right] \cup \left({4a^2b^3}, {180.45b^3}/{a}\right). 
$$

Note that the last  interval is not empty if and only if $1\le a \le 3$. 

\vspace{5mm}

\noindent $\bullet$ \textbf{Case I:} $4ab<c\le 4a^{\frac{1}{2}}b^{\frac{3}{2}}.$  
Since $c=d_{+}(a,b,d_{-1})$, we have that $c>4abd_{-1}$ and therefore, we obtain
$$
ad_{-1}<\frac{c}{4b}<(ab)^{\frac{1}{2}},
$$
and in particular we have that $ab>(ad_{-1})^2$. Since $c>4abd_{-1}$ we get $ac>4(ab)(ad_{-1})$ and moreover 
$$
ad_{-1}<\left(\frac{ac}{4}\right)^{\frac{1}{3}}<\left(\frac{6.77\cdot 10^{25}}4\right)^{\frac{1}{3}}< 256749472.
$$
 
Put $r_{(a,d_{-1})} = \sqrt{ad_{-1}+1}$. Since $ad_{-1}+1$ is a perfect square, $r_{(a,d_{-1})}$ is a positive integer with $2\le r_{(a,d_{-1})}\le 16023$.
Using a short GP program, we see that there are 1081908 pairs $(a,d_{-1})$ to be checked in this range. Note that $\{a,d_{-1},b\}$ is
a Diophantine triple. For a fixed pair $(a,d_{-1})$, there exist positive integers $U=r= \sqrt{ab+1}$ and $V=\sqrt{bd_{-1}+1}$ such
that $b=\frac{V^2-1}{d_{-1}}=\frac{U^2-1}{a}$ and $\max\{U,V\}\le b^{\frac{1}{2}}$. Indeed, we have 
$$4a^2b<a \cdot 4ab<ac<6.77\cdot 10^{25}$$
and therefore we may assume that $\max\{U,V\}\le 4.12\cdot 10^{12}$.
  
In order to find all possible values of $b$, we consider the Pell equation   
\begin{equation}\label{eq:pell}
\mathcal{A}\mathcal{V}^2 - \mathcal{B}\mathcal{U}^2 = \mathcal{A}-\mathcal{B}, 
\end{equation}
where $\mathcal{AB} + 1= \mathcal{R}^2$, $0<\mathcal{A,B,R}\in \ZZ$ and $\mathcal{A} < \mathcal{B}$.
By Lemma~\ref{lem:boundsx_0y_0}, all positive integer solutions to the above Pell equation can be determined by 
$$
\mathcal{V}\sqrt{\mathcal{A}} + \mathcal{U}\sqrt{\mathcal{B}}=\mathcal{V}_q\sqrt{\mathcal{A}} + \mathcal{U}_q\sqrt{\mathcal{B}} =
(\mathcal{V}_0\sqrt{\mathcal{A}}+\mathcal{U}_0\sqrt{\mathcal{B}})(\mathcal{R}+ \sqrt{\mathcal{AB}})^q, \quad q\ge 0,
$$
where $(\mathcal{V}_0,\mathcal{U}_0)$ is a fundamental solution to Pell equation \eqref{eq:pell}, hence satisfies
$$
0\le|\mathcal{V}_0|\le \sqrt{\frac{1}{2}\mathcal{A}(\mathcal{B}-\mathcal{A})(\mathcal{R}-1)},\quad 0<\mathcal{U}_0\le \sqrt{\frac{\mathcal{A}(\mathcal{B}-\mathcal{A})}{2(\mathcal{R}-1)}}.
$$

Our program runs over all $r_{(a,d_{-1})}$ in the range from $2$ to $16023$. For each $\mathcal{R}= r_{(a,d_{-1})}$,
we consider the divisors $d'$ of $ \mathcal{R}^2-1$ with $1\le d'\le \mathcal{R}$ and put $\mathcal{A}=d'$, $\mathcal{B}=(\mathcal{R}^2-1)/\mathcal{A}$.  For  each pair $(\mathcal{A},\mathcal{B})$, we find all possible fundamental solutions $(\mathcal{V}_0,\mathcal{U}_0)$ to 
equation~\eqref{eq:pell} and consider the corresponding sequences $\mathcal{U}_q$. Notice that not all solutions $\mathcal{U}$ of \eqref{eq:pell}
satisfy $\mathcal{A}|(\mathcal{U}^2-1)$. If $\mathcal{A}|(\mathcal{U}^2-1)$ and $\mathcal U= \mathcal{U}_q<4.12\cdot 10^{12}$,
then we put $(a,d_{-1},b)$ or $(d_{-1},a,b)=(\mathcal{A},\mathcal{B},\mathcal{C})$ and $c=d_{+}(a,d_{-1},b)$, where $\mathcal C=\frac{\mathcal U^2-1}{\mathcal A}$. 

Applying Lemma~\ref{lem:Baker-Davenport} 
to $\Lambda_1$, we checked all $2340242$ triples possible $(a,b,c)$ in $15$ minutes with our GP program. In all cases we obtain that $J\le6$, which is impossible as $J>20\sqrt{2}$.

\vspace{5mm}

\noindent $\bullet$ \textbf{Case II :}  $4a^{\frac{1}{2}}b^{\frac{3}{2}}<c\le4ab^2$. Since $c=d_{+}(a,b,d_{-1})$, we deduce from our assumption that
$d_{-1}<\frac{c}{4ab}<b$. This implies that $b=\max\{a,b,d_{-1}\}$ and Lemma \ref{lem:d+ieq} yields $c<4b(ad_{-1}+1)$.
On the other hand, we have by our assumption that $4a^{\frac{1}{2}}b^{\frac{3}{2}}<c$ and therefore we get $(ab)^{\frac{1}{2}}-1<ad_{-1}$.  
Thus, we get
$$
d_{-2}=d_{-}(a,b,d_{-1}) <\frac{b}{4ad_{-1}}<\frac{b}{4\left((ab)^{\frac{1}{2}}-1\right)}.
$$
and 
$$
4ad_{-2} < \frac{ab}{(ab)^{\frac{1}{2}}-1}< (ab)^{\frac{1}{2}}+2.
$$
From $(4ad_{-2}-2)^2<ab<(ad_{-1}+1)^2$ we have $4ad_{-2}<ad_{-1}+3$ and so $d_{-2}<d_{-1}$. 
Substituting $ab>(4ad_{-2}-2)^2$ into the inequality $ac>4(ab)(ad_{-1})>4(ab)(ad_{-2})$,  we obtain
$$
4(4ad_{-2}-2)^2(ad_{-2})<ac<6.77\cdot 10^{25}. 
$$
It follows that $ad_{-2}<101891096$ and we get that $r_{(a,d_{-2})}=\sqrt{ad_{-2}+1}\le 10095$. Moreover, we 
know that 
$$d_{-1}<b<\left(\frac{6.77\cdot10^{25}}4\right)^{2/3}<6.6\cdot 10^{16}.$$
Similarly to Case I,  we  put $(a,d_{-2})$ or $(d_{-2},a) = (\mathcal{A}, \mathcal{B})$ in equation \eqref{eq:pell}.
And we set $(a,d_{-2},d_{-1})$ or $(d_{-2},a,d_{-1})=(\mathcal{A},\mathcal{B},(\mathcal{U}_q^2-1)/\mathcal{A})$ 
when $\mathcal{A}|\mathcal{U}_q^2-1$, for $\mathcal{U}_q<b^{1/2}<2.57\cdot 10^8$. Using $b=d_{+}(a,d_{-1},d_{-2})$ and
$c=d_{+}(a,b,d_{-1})$, we apply Lemma~\ref{lem:Baker-Davenport} to $\Lambda_1$ and check
$2565234$ triples $(a,b,c)$. The verification with our GP program took $20$ minutes and we obtained that $J\le6$ in all cases. But this is also impossible as $J>20\sqrt{2}$.

\vspace{5mm}

\noindent $\bullet$ \textbf{Case III:} $4ab^{2}<c\le4a^{\frac{3}{2}}b^{\frac{5}{2}}$. First, we observe that the inequality $4a^2b^2 <ac<6.77\cdot 10^{25}$ yields the upper bound $r<2028300$. Since $b^{1/2}<\sqrt{ab+1}=r$, we also have an upper bound for $b^{1/2}$.

Assume for the moment that $b>d_{-1}$, then $b=d_{+}(a,d_{-1},d_{-2}) > 4ad_{-1}$ and therefore $d_{-1}<\frac{b}{4a}$.
But this yields
$$4ab^2<c<4abd_{-1}+4b<b^2+4b,$$
which is impossible. Therefore, we may assume that $b<d_{-1}$. Since $d_{-1}<\frac{c}{4ab}$, we get 
$$
d_{-1} < a^{\frac12}b^{\frac32}. 
$$
and therefore
$$
ad_{-2} <\frac{d_{-1}}{4b}<\frac{(ab)^{\frac12}}4<\frac r4<507075.
$$

If we write $r_{(a,d_{-2})} = \sqrt{ad_{-2}+1}$, then we have that $r_{(a,d_{-2})}\le 712$.
Using the algorithm of Case I, we set $(a,d_{-2})$ or $(d_{-2},a) = (\mathcal{A}, \mathcal{B})$ in equation \eqref{eq:pell}.
We take $(a,d_{-2},b)$ or $(d_{-2},a,b)=(\mathcal{A},\mathcal{B},(\mathcal{U}_q^2-1)/\mathcal{A})$ when $\mathcal{A}|\mathcal{U}_q^2-1$, 
for $\mathcal{U}_q<b^{1/2}<2028300$. Using $d_{-1}=d_{+}(a,d_{-2},b)$ and $c=d_{+}(a,b,d_{-1})$, we apply  Lemma~\ref{lem:Baker-Davenport}
to $\Lambda_1$ and check all $102032$ triples $(a,b,c)$ with our GP program, which took $1$ minute and $15$ seconds. In all cases we obtain $J\le 14$. This contradicts the fact that  $J>20\sqrt{2}$.

\vspace{5mm}

\noindent $\bullet$ \textbf{Case IV:} $4a^{\frac{3}{2}}b^{\frac{5}{2}}<c\le4a^2b^3$. By the same arguments as in Case III we may assume that $b<d_{-1}$. Moreover, we have
$$
d_{-1} < \frac{c}{4ab} <  ab^2 \quad \mbox{and}\quad 
d_{-2} < \frac{d_{-1}}{4ab}<\frac{b}{4}.
$$
The later inequalities imply that $\{a,d_{-2},b\}$ is not an Euler triple. Therefore, there exists a positive integer $$d_{-3}= d_{-}(a,d_{-2},b).$$  
Let's estimate the upper bound of $ad_{-3}$. Using Lemma~\ref{lem:d+ieq} to the regular Diophantine quadruple $\{a,b,d_{-1},c\}$, we have $c<4abd_{-1}+4d_{-1}$. It follows that 
$$
d_{-1}>\frac{c}{4ab+4}>\frac{a^{\frac{3}{2}}b^{\frac{5}{2}}}{ab+1}.
$$
Again here, the regular  Diophantine quadruple $\{a,d_{-2},b,d_{-1}\}$ provides $d_{-1}<4abd_{-2}+4b$, and so
$$
\quad d_{-2}>\frac{d_{-1}-4b}{4ab}.
$$
With $b>4ad_{-2}d_{-3}$, we have 
$$
ad_{-3}<\frac{b}{4d_{-2}}<\frac{ab^2}{d_{-1}-4b}<\frac{ab}{\frac{(ab)^{\frac{3}{2}}}{ab+1}-4}. 
$$
On the other hand, $4(ab)^{5/2}<ac<6.77\cdot 10^{25}$ yields $ab<1.24\cdot 10^{10}$. Therefore, we get
$$
ad_{-3} < 111360.  
$$

Let us write $r_{(a,d_{-3})}=\sqrt{ad_{-3}+1}$, then we get that $r_{(a,d_{-3})}\le 333$.
There are $8854$ pairs $(a,d_{-3})$ satisfying
this inequality. We set $(a,d_{-3})$ or $(d_{-3},a) = (\mathcal{A}, \mathcal{B})$ in equation \eqref{eq:pell}.
For solutions $\mathcal{U}_q$ to Pell equation \eqref{eq:pell} we put $(a,d_{-3},d_{-2})$ or $(d_{-3},a,d_{-2})=(\mathcal{A},\mathcal{B},(\mathcal{U}_q^2-1)/\mathcal{A})$,
if $\mathcal{A}|\mathcal{U}_q^2-1$. Note that we may assume that $\mathcal{U}_q<b^{1/2}<111356$ since $b\leq ab<1.24\cdot 10^{10}$.
Furthermore we compute $b=d_{+}(a,d_{-3},d_{-2})$, $d_{-1}=d_{+}(a,d_{-2},b)$ and $c=d_{+}(a,b,d_{-1})$ and
apply Lemma~\ref{lem:Baker-Davenport} to $\Lambda_1$. Our GP program checked all $36762$ triples $(a,b,c)$ in $26$ seconds and we got $J\le 6$ in each case.
This also contradicts the fact $J>20\sqrt{2}$.

\vspace{5mm}

\noindent $\bullet$ \textbf{Case V:}  $ 4a^2b^3<c\le \frac{180.45b^3}{a}$. If this interval is nonempty, then we have $4a^3<180.45$. It follows that $1\le a\le3$. It is clear that $b<d_{-1}$. By Lemma~\ref{lem:d+ieq}, we have 
\begin{equation}\label{eq:case5-1}
\frac{a^2b^3}{ab+1}<\frac{c}{4(ab+1)}<d_{-1} < \frac{c}{4ab}<\frac{45.2b^2}{a^2}. 
\end{equation}
We have $d_{-2}=d_{-}(a,b,d_{-1})<\frac{d_{-1}}{4ab}<\frac{11.3b}{a^3}$.   

If the Diophantine triple $\{a, b, d_{-2}\}$ is not an Euler triple, then there exists a positive integer $d_{-3}=d_{-}(a,b,d_{-2})$. When $b<d_{-2}$,  we have $d_{-2}>4abd_{-3}\ge 12b$. This and $d_{-2}<\frac{11.3b}{a^3}$ provides $12a^3<11.3$, which is impossible.  When $b>d_{-2}$, then  $b>4ad_{-2}d_{-3}\ge 12d_{-2}$. Since $d_{-1}<4abd_{-2}+4b$, then we have 
$d_{-2}>\frac{d_{-1}-4b}{4ab}$, and so $b>12d_{-2} > \frac{3d_{-1}-12b}{ab}$. It follows that $d_{-1}<\frac{ab^2+12b}{3}$. This and \eqref{eq:case5-1} gives 
$$
\frac{a^2b^3}{ab+1}<\frac{ab^2+12b}{3}. 
$$
It follows that $2a^2b^2-13ab -12<0$. We have $ab<8$. Then $\{a,b\}=\{1,3\}$ or $\{2,4\}$. But no integer $d_{2}$ less than $b$ is such that $\{a,d_{-2},b\}$ is a Diophantine triple. 

Then we know that the Diophantine triple $\{a, b, d_{-2}\}$ is an Euler triple.  For $1\le a\le 3$ and $r\le 16023$, we have 
$$
b=\frac{r^2-1}{a},\quad d_{-2}=a+b\pm2r= a+\frac{r^2-1}{a}\pm 2r.
$$
Moreover, we compute $d_{-1}=d_{+}(a,d_{-2},b)$ and $c=d_{+}(a,b,d_{-1})$.  We checked $69428$ triples $(a,b,c)$ in  
$40$ seconds, and we have $J\le 16$. This contradicts $J>20\sqrt{2}$ again.

This completes the proof of Theorem~\ref{thm:deg2}.
\end{proof}

\section{Proof of Theorem~\ref{thm:main}}\label{sec:9}

Assume that $\{a,b,c,d,e\}$ is a Diophantine quintuple, with $a<b<c<d<e$. By Theorem \ref{thm:fujita} (cf. \cite{Fujita-1}), the quadruple 
$\{a,b,c,d\}$ is regular, i.e. the element $d$ is uniquely determined by $a,b$ and $c$.  By Proposition~\ref{pro:2-3},
for an arbitrary but fixed Diophantine triple $\{a,b,c\}$, there exists a nonnegative integer $D=\deg(a,b,c)$ such that $\{a,b,c\}$
is generated by some Euler triple $\{a',b',c'\}$. Theorem~\ref{thm:euler}, Theorem~\ref{thm:eulerplus} and Theorem~\ref{thm:deg2}
show that a Diophantine triple $\{a,b,c\}$ cannot be extended to Diophantine quintuple $\{a,b,c,d,e\}$ for $D=0,1$ and $D\ge 2$, 
respectively. This completes the proof of Theorem~\ref{thm:main}.

\section*{Acknowledgements} The first author was supported by Natural Science Foundation of China (Grant No. 11301363), 
and Sichuan provincial scientific research and innovation team in Universities (Grant No. 14TD0040), and the Natural Science 
Foundation of Education Department of Sichuan Province (Grant No. 16ZA0371). The second author thanks Purdue University Northwest
for the partial support. The third author was supported in part by the Austrian Science Fund (FWF) (Grant No. P 24801-N26). 

The authors also want to thank the organizers of the {\it Diophantine m-tuples and related problems} conference (in 2014) and the {\it 29th Journ\'ees Arithm\'etiques }(in 2015). At both conferences, they had the opportunity to discuss the ideas that led to this paper. We also want to thank Mihai Cipu and Marija Bliznac Trebje\v{s}anin for their valuable comments and for pointing out several flaws and typos in an earlier version of the paper.

Finally, they are grateful to referees whose numerous comments help to seriously improve this paper.

\end{document}